\documentclass[11pt,leqno]{article}
\usepackage[utf8]{inputenc}
\usepackage{amsmath}
\usepackage{amsfonts}
\usepackage{amssymb}
\usepackage{graphicx}
\usepackage{mathrsfs}

\usepackage[hmargin=2.75cm, vmargin=2.9cm]{geometry}

\usepackage{upref,amsthm,amsxtra,exscale}
\usepackage{cite}
\usepackage[colorlinks=true,urlcolor=blue,
citecolor=red,linkcolor=blue,linktocpage,pdfpagelabels,
bookmarksnumbered,bookmarksopen]{hyperref}
\allowdisplaybreaks[4]

\newtheorem{theorem}{Theorem}[section]

\newtheorem{remark}[theorem]{Remark}
\newtheorem{lemma}[theorem]{Lemma}
\newtheorem{proposition}[theorem]{Proposition}

\numberwithin{equation}{section}

\usepackage{color,graphicx,xcolor}

\newcommand{\al}{\alpha}
\newcommand{\de}{\delta}

\newcommand{\eps}{\varepsilon}
\newcommand{\R}{{\mathbb R}}

\def\sideremark#1{\ifvmode\leavevmode\fi\vadjust{\vbox to0pt{\vss
\hbox to 0pt{\hskip\hsize\hskip1em
\vbox{\hsize2.1cm\tiny\raggedright\pretolerance10000
\noindent #1\hfill}\hss}\vbox to15pt{\vfil}\vss}}}%

\def\d{\,\mathrm{d}}
\def\e{\mathrm{e}}
\def\r{\mathbb{R}}
\def\rn{\mathbb{R}^N}

\def\eps{\varepsilon}

\def\io{\int_{\Omega}}
\def\irn{\int_{\r^N}}

\author{Mónica Clapp\footnote{M. Clapp was partially supported by UNAM-DGAPA-PAPIIT grant IN100718 and CONACYT grant A1-S-10457 (Mexico).}, Rosa Pardo\footnote{R. Pardo was partially supported by Unión Iberoamericana de Universidades (UIU) ref. UCM-04-2019, by  grants MTM2016-75465-P, and PID2019-103860GB-I00,  MICINN,  Spain, and by UCM-BSCH, Spain, GR58/08, Grupo 920894.}, Angela Pistoia
\footnote{A. Pistoia was partially
supported by Fondi di Ateneo “Sapienza” Universit`a di Roma (Italy) and project Vain-Hopes within the
program VALERE: VAnviteLli pEr la RicErca.}, and Alberto Saldaña\footnote{A. Saldaña was partially supported by the Alexander von Humboldt Foundation (Germany).}}
\title{A solution to a slightly subcritical elliptic problem with non-power nonlinearity}
\date{\today}

\begin{document}

\maketitle

\begin{abstract}
We consider a slightly subcritical Dirichlet problem with a non-power nonlinearity in a bounded smooth domain.  For this problem, standard compact embeddings cannot be used to guarantee the existence of solutions as in the case of power-type nonlinearities.  Instead, we use a Ljapunov-Schmidt reduction method to show that there is a positive solution which concentrates at a non-degenerate critical point of the Robin function. This is the first existence result for this type of generalized slightly subcritical problems.

\medskip

\medskip

\end{abstract}

\section{Introduction}

We consider the problem
\begin{equation}\label{Plog}
\begin{cases}
-\Delta u =f_\eps(u)&\text{ in }\Omega,\\
u=0&\text{ on }\partial\Omega,
\end{cases}
\end{equation}
where $\Omega$ is a bounded smooth domain in $\R^N$, $N\geq 3$, and
\begin{align}\label{flog}
f_\eps(u):=\frac{|u|^{2^*-2}u}{[\ln(\e+|u|)]^\eps},\quad \eps\geq 0.
\end{align}
Here, $2^*:=\frac{2N}{N-2}$ is the critical Sobolev exponent.

If $\eps=0$, then \eqref{Plog} is called the pure critical problem. In this case, the existence of solutions is strongly affected by the geometry of the domain.
Indeed, Pohozaev's identity~\cite{po} ensures the non-existence of solutions in   star-shaped domains, while the existence of a positive solution was established by Bahri and Coron ~\cite{bc}  in a domain with non-trivial topology.

Most of the analysis of slightly subcritical problems has been focused on power-type nonlinearities ($|u|^{2^*-2-\eps}u$ instead of $f_\eps(u)$).  However, if one considers problems with other types of subcritical behavior (such as \eqref{Plog} with \eqref{flog}), then many of the techniques developed for the power nonlinearity cannot be applied anymore.  For example, one cannot use directly the compactness of Sobolev embeddings to guarantee the convergence of Palais-Smale sequences associated to \eqref{Plog}. Another well-known approach to find solutions of elliptic problems is to find a uniform a priori bound and establish an existence result using Leray-Schauder degree theory. However, if $\eps>0$, then the existence of a uniform a priori bound for the $L^\infty$-norm of all positive solutions to the problem~\eqref{Plog} is, in general, not known. The classical results of Gidas and Spruck~\cite{gs} and de Figueiredo, Lions, and Nussbaum~\cite{dln} do not apply to this problem.  In this direction, some progress has been made recently. In~\cite{cp}, Castro and Pardo obtained a priori bounds for nonlinearities including \eqref{flog} with $\eps>\frac{2}{N-2}$, which are not covered by~\cite{gs, dln}. The arguments rely on the moving plane method (providing uniform a priori bounds in a neighborhood of the boundary), the Pohozaev identity, $W^{1,q}$  regularity  for $q>N$, and  Morrey's theorem.  Using the Kelvin transform, they  extend the existence of uniform a priori bounds to non-convex domains, see~\cite{cp,cp-rl}.  These results are, however, only available for $\eps>\frac{2}{N-2}$, and do not include slightly subcritical problems. 

We believe that the study of problems such as \eqref{Plog} with \eqref{flog}, for $\eps>0$ small, improves our understanding of more general subcritical problems and helps to develop more flexible and powerful tools in nonlinear analysis.

\medskip

In this paper, we establish the existence of a solution to~\eqref{Plog} which blows-up at a point in $\Omega$ when $\eps\to0$.  First, let us introduce the so-called standard bubbles
\begin{equation}\label{bub:intro}
U(y):=\alpha_N   \frac{1}{(1+|y|^2)^{\frac{N-2}{2}}},\qquad
U_{\delta,\xi}(x)=\delta^{-\frac{N-2}2}U\left(\frac{x-\xi}{\delta}\right),\quad  \delta>0,\;\xi\in\rn,
\end{equation}
where $\alpha_N=[N(N-2)]^\frac{N-2}{4}$.  Next,  let $G$ be the Green function of $-\Delta$ in $\Omega$ with Dirichlet boundary condition, and let $H$ be its regular part, \textit{i.e.},
$$G(x,y)={c_N}\left(\frac1{|x-y|^{N-2}}-H(x,y)\right),\qquad x,y\in \Omega,$$
where $c_N=\frac1{(N-2)\omega_{N}}$ and $\omega_N$ denotes the surface area of the unit sphere in $\mathbb R^N.$ The function $\varrho :\Omega\to\r$ given by 
$$\varrho (x):=H(x,x)$$
is called the Robin function.  Our main result is the following one.

\begin{theorem}\label{main:thm}
Let $\xi^*\in \Omega$ be a non-degenerate critical point of the Robin function.
Then, there exists a solution to \eqref{Plog} which blows up at $\xi^*$ as $\eps\to0.$ More precisely, there exists $\eps_0>0$ such that, for each $\eps\in(0,\eps_0)$, there is a solution $u_\eps\in H^1_0(\Omega)$ of~\eqref{Plog} of the form
\begin{align} \label{uepsbd}
u_\eps=U_{\delta(\eps),\xi(\eps)}+\Phi_\eps,\quad \hbox{ where }\quad \int_\Omega|\nabla \Phi_\eps |^2=\mathcal O\left(\frac{ {\eps}}{{|\ln\eps|}}\right)
\end{align}
and with $U_{\delta(\eps),\xi(\eps)}$ as in \eqref{bub:intro}. The concentration parameter $\delta(\eps)$ and the blow-up point $\xi(\eps)$ satisfy
$$
\delta(\eps) \left(\frac{|\ln\eps|}{\eps}\right)^{\frac{1}{N-2}}\to d>0
\quad \hbox{and}\quad
\xi(\eps) \to\xi^*\quad \hbox{as}\ \eps\to 0.$$
\end{theorem}

This seems to be the first existence result for problem~\eqref{Plog} when $\eps>0$ is arbitrarily small.  We point out that in any domain $\Omega$ the Robin function has at least one critical point which is a minimum point (since it tends to infinity at $\partial\Omega$) and also that the  minimum value is   strictly positive. 
Moreover, Micheletti and Pistoia    in~\cite[Theorem 1.1]{mp} proved  that, for almost every domain $\Omega$, the Robin function is a Morse function, i.e., all its critical points are non-degenerate.  It is also known that the origin is a non-degenerate critical point of the Robin function of a smooth bounded domain of $\R^N$ which is symmetric with respect to the origin and convex in any direction $x_1,\ldots,x_N$, as proved by Grossi in~\cite{g}.

Theorem \ref{main:thm} is shown using the Ljapunov-Schmidt reduction method. One of the advantages of this approach is that we obtain explicit information about the behavior of the solution.  In particular, it is interesting to compare the blow-up rate of the solution $u_\eps$ given by Theorem~\ref{main:thm} ($\|u_\eps\|_{\infty}\sim (|\ln\eps|\eps^{-1})^{\frac{1}{2}}$) with the blow-up rate $\eps^{-\frac{1}{2}}$ associated with the usual power nonlinearity, as shown by Bahri and Rey~\cite{blr}.

Theorem \ref{main:thm} is a first step towards establishing existence and multiplicity of positive and/or sign-changing solutions to problem \eqref{Plog} which blow up and/or blow down at different points in $\Omega$ as  $\eps\to0.$  This is motivated by a series of results which have been obtained in the last decades in the subcritical regime with power-type nonlinearities, namely, when the nonlinear term $f_\eps(u)$ is replaced by $|u|^{2^*-2-\eps}u$ with $\eps>0.$ In this case  the compactness of the Sobolev embedding $H^1_0(\Omega)\hookrightarrow L^{2^*-\eps}(\Omega)$ yields the existence of a  least-energy solution to
\begin{equation}\label{Psc}
\begin{cases}
-\Delta u =|u|^{2^*-2-\eps}u&\text{ in }\Omega,\\
u=0&\text{ on }\partial\Omega,
\end{cases}
\end{equation}
by standard variational methods. Han in~\cite{h} proved that as $\eps\to0$  this solution blows up at a  point $\xi_0\in\Omega$  and  its limit profile is a rescaling of the standard bubble  \eqref{bub:intro}.
Flucher and Wei in~\cite{fw} proved that $\xi_0$ is the minimum of the Robin function.
The existence of positive solutions  of \eqref{Psc} which blow up at different points in $\Omega$ was studied by Bahri, Li, and Rey in~\cite{blr} using a finite dimensional reduction procedure. A similar argument was used by Bartsch, Micheletti, and Pistoia in~\cite{bmp}  to prove the existence of 
sign-changing solutions of \eqref{Psc} which blow up or blow down at different points in $\Omega$. 
In both cases the location of the blow-up and blow-down point is given in terms of a reduced energy which involves the Green and the Robin function.  Finally, we also recall that sign-changing solutions of \eqref{Psc},  which  blow up and down at the same point (sometimes called the nodal towering point) have been found by Pistoia and Weth in~\cite{pw}.  In particular, Musso and Pistoia in~\cite{mup}  proved that any non-degenerate critical point of the Robin function is a nodal towering point. 

We  conjecture that similar results as those obtained in \cite{bmp,pw} can be extended to problem \eqref{Plog} with \eqref{flog}, however the proof requires some careful estimates to overcome the essential technical difficulties due to the strong nonlinearity \eqref{flog}.

\medskip

To close this introduction, we mention that in~\cite{dp}, Damascelli and Pardo found a priori bounds for the $p$-Laplacian version of \eqref{Plog}. Furthermore, the existence of uniform a priori bounds and, thus, of a positive solution to the Hamiltonian elliptic system 
\begin{equation}\label{Ham}
-\Delta u=v^{p}\big/[\ln (e+v)]^{\eps_1},\quad -\Delta v=u^{q}\big/[\ln (e+u)]^{\eps_2},\quad\text{in }\Omega,\qquad u=v= 0\quad\text{on }\partial \Omega, 
\end{equation}
with $\min \{\eps_1,\, \eps_2\}> 2/(N-2) $, and  $p,q$  lying in the  critical Sobolev hyperbola $\frac{1}{p+1} +\frac{1}{q+1} = \frac{N-2}{N}$,  was treated by Mavinga and Pardo in ~\cite{mp2017}.

\medskip

The paper is organized as follows.  In Section~\ref{sec:pre} we include some notation and well-known results regarding the Ljapunov-Schmidt method. In Section~\ref{sec:fdr} we structure the finite dimensional reduction for the problem, some of the results in this section are well known, but we include a proof for clarity and completeness. Finally, in Section~\ref{sec:fdp}, we find a critical point of the reduced problem. We close the paper with an appendix containing some useful estimates associated to the nonlinearity~\eqref{flog}.

\section{Preliminaries}\label{sec:pre}

Consider the Hilbert space $D^{1,2}(\rn):=\{u\in L^{2^*}(\rn):\nabla u\in L^2(\rn,\rn)\}$ with its usual inner product and norm,
$$\langle u,v\rangle:=\irn\nabla u\cdot\nabla v,\qquad \|u\|:=\left(\irn |\nabla u|^2\right)^{1/2}.$$
It is well known that the standard bubbles 
\begin{equation}\label{bub}
U(y):=\alpha_N   \frac{1}{(1+|y|^2)^{\frac{N-2}{2}}},\qquad
U_{\delta,\xi}(x)=\delta^{-\frac{N-2}2}U\left(\frac{x-\xi}{\delta}\right),\quad  \delta>0,\;\xi\in\rn,
\end{equation}
are the only positive solutions of the problem
\begin{align} \label{eq:rn}
-\Delta u = |u|^{2^*-2}u,\qquad u\in D^{1,2}(\rn),
\end{align}
where $\alpha_N=[N(N-2)]^\frac{N-2}{4}$. They satisfy
$$\|U_{\delta,\xi}\|^2=|U_{\delta,\xi}|_{2^*}^{2^*}=S^\frac{N}{2},$$
where $S$ is the best constant for the embedding $D^{1,2}(\rn)\hookrightarrow L^{2^*}(\rn)$ and $|\cdot|_p$ is the standard $L^p$-norm.

Set $p:=2^*-1$. It is well known that the kernel of the linearized operator $-\Delta-pU ^{p-1}{\mathtt I}$, \textit{i.e.}, the space of solutions to the problem
\begin{equation} \label{eq:linearized}
-\Delta \psi = pU ^{p-1}\psi,\qquad \psi\in D^{1,2}(\rn),
\end{equation}
is generated by the $N+1$ functions 
\begin{equation}\label{psi}\begin{aligned}
\psi ^0(y) &:= \frac{N-2}{2}\alpha_N \frac{|y|^2-1}{(1+|y|^2)^{N/2}},\\
\psi ^j(y) &: =(N-2)\alpha_N  \frac{y_j}{(1+|y|^2)^{N/2}},\quad j=1,\ldots,N.
\end{aligned}\end{equation}
Set
\begin{equation}\label{ppsi}
\begin{aligned}
\psi_{\delta,\xi}^0(x) &:=\delta^{-\frac{N-2}{2}}\psi^0\left(\frac{x-\xi}{\delta}\right)=\frac{N-2}{2}\alpha_N\,\delta^{\frac{N-2}{2}}\frac{|x-\xi|^2-\delta^2}{(\delta^2+|x-\xi|^2)^{N/2}},\\
\psi_{\delta,\xi}^j(x) &:=\delta^{-\frac{N-2}{2}}\psi^j\left(\frac{x-\xi}{\delta}\right)=(N-2)\alpha_N\,\delta^{\frac{N}{2}}\frac{x_j-\xi_j}{(\delta^2+|x-\xi|^2)^{N/2}},\quad j=1,\ldots,N.
\end{aligned}
\end{equation}
In particular,
\begin{align*}
\psi_{\delta,\xi}^0 =\de\,\frac{\partial U_{\delta,\xi}}{\partial\delta},\qquad
\psi_{\delta,\xi}^j=\de\,\frac{\partial U_{\delta,\xi}}{\partial\xi_j},\quad j=1,\ldots,N.
\end{align*}
Note that, as $U_{\delta,\xi}$ solves~\eqref{eq:rn}, any solution $\psi$ to~\eqref{eq:linearized} satisfies
$$\irn U_{\delta,\xi}^p\psi=\irn\nabla U_{\delta,\xi}\cdot\nabla\psi=p\irn U_{\delta,\xi}^p\psi.$$
In particular, 
\begin{equation} \label{eq:Upsi}
\langle U_{\delta,\xi},\psi^j_{\delta,\xi}\rangle = \irn U_{\delta,\xi}^p\psi^j_{\delta,\xi}=0\qquad\forall j=0,\ldots,N.
\end{equation}

We denote by $P:D^{1,2}(\rn)\to H^1_0(\Omega)$ the orthogonal projection, \textit{i.e.},  $PW$ is the unique solution to the problem
\begin{align*}
-\Delta(PW) = -\Delta W\quad \text{ in }\Omega,\qquad PW = 0\quad \text{ on }\partial \Omega.
\end{align*}

Next, we collect some well known estimates.
\begin{lemma}\label{lem:rey}
The following expansions hold true
\begin{align}
&P U_{\delta,\xi} =U_{\delta,\xi}- \alpha_N\,\delta^{\frac{N-2}{2}}H(\,\cdot\,,\xi) +\mathcal O(\delta^\frac{N+2}{2}),\label{e1}\\
&P\psi^0_{\delta,\xi}=\psi^0_{\delta,\xi}-\frac{N-2}{2}\alpha_N\delta^\frac{N-2}{2}H(\cdot,\xi)+ \mathcal  O(\delta^\frac{N+4}{2}),\label{e2}\\
&P\psi^j_{\delta,\xi}=\psi^j_{\delta,\xi}-\alpha_N\delta^\frac{N }{2}\partial_{\xi_j}H(\cdot,\xi)+ \mathcal  O(\delta^\frac{N+2}{2})\quad\text{for }j=1,\ldots,N,\label{e3}
\end{align}
as $\delta\to0$ uniformly with respect to $\xi$ in compact subsets of $\Omega$.
Moreover,
\begin{equation}\label{l2}|P\psi_{\delta,\xi}^j -\psi_{\delta,\xi}^j|_\frac{2N}{N-2}  =
\begin{cases}
\mathcal O\big(\delta^\frac{N-2}{2}\big)\ &\hbox{if}\ j=0,\\
\mathcal O\big(\delta^\frac{N}{2}\big)\ &\hbox{if}\ j=1,\dots,N, \end{cases}
\end{equation}
and
\begin{equation}\label{l3}\langle P\psi_{\delta,\xi}^i,P\psi_{\delta,\xi}^j\rangle=
\begin{cases}
c_i(1+o(1))>0\ &\hbox{if}\ i=j\\
o(1)\ &\hbox{if}\ i\not=j. 
\end{cases}
\end{equation} 
\end{lemma} 
\begin{proof}
  For the proof of \eqref{e1}, \eqref{e2}, and \eqref{e3}, see \cite[Proposition 1]{r}.  Then \eqref{l2} follows from \eqref{e2} and \eqref{e3}. For \eqref{l3} we argue as in \cite[Lemma 3.1]{pt}. Observe that, by \eqref{e2} and \eqref{e3},
 \begin{align*}
  \langle P\psi^{i}_{\delta,\xi},P\psi^j_{\delta,\xi} \rangle
  =\io f_0'(U_{\delta,\xi})\psi^{i}_{\delta,\xi}P\psi^j_{\delta,\xi}
  =\io f_0'(U_{\delta,\xi})\psi^{i}_{\delta,\xi}\psi^j_{\delta,\xi}+o(1),
 \end{align*}
for $i,j=0,\ldots,N$ as $\delta\to 0$. Then, changing variables,
\begin{align*}
 \io f_0'(U_{\delta,\xi})\psi^{i}_{\delta,\xi}\psi^j_{\delta,\xi}
 =(2^*-1)\int_{\R^N} |U|^{2^*-2}\psi^{i}\psi^j+o(1),
\end{align*}
and \eqref{l3} follows by oddness, because
\begin{align*}
\psi^j(y)=
 \begin{cases}
\frac{N-2}{2}\alpha_N\frac{|y|^2-1}{(1+|y|^2)^{N/2}}, &j=0,\\  
(N-2)\alpha_N\frac{y_j}{(1+|y|^2)^{N/2}}, &j=1,\ldots,N.
\end{cases}
\end{align*}
\end{proof}

We set 
\begin{align*}
K_{\delta,\xi}&:=\mathrm{span}\{P\psi_{\delta,\xi}^j: j=0,\ldots,N\},\\
K_{\delta,\xi}^\perp &:=\{\phi\in H^1_0(\Omega):\langle\phi,P\psi_{\delta,\xi}^j\rangle=0,\;j=0,\ldots,N\},
\end{align*}
and denote by
$$\Pi_{\delta,\xi}:H^1_0(\Omega)\to K_{\delta,\xi}\qquad\text{and}\qquad \Pi_{\delta,\xi}^\perp:H^1_0(\Omega)\to K_{\delta,\xi}^\perp$$
the orthogonal projections.

\section{The finite dimensional reduction}\label{sec:fdr}

To prove our main result we apply the well-known Ljapunov-Schmidt reduction procedure; see~\cite{p} and the references therein for a detailed discussion of this approach.

Let $i^*:L^{\frac{2^*}{2^*-1}}(\Omega)\to H^1_0 (\Omega)$ be the adjoint operator of the embedding
$i:H^1_0 (\Omega)\hookrightarrow L^{2^*}(\Omega)$, \textit{i.e.}, $i^*[v]$ is the unique solution to the problem
\begin{equation*}
-\Delta u=v\,\text{ in }\Omega,\qquad u=0\,\text{ on }\partial\Omega.
\end{equation*}
It is well known that $i^*$ is a continuous map and
\begin{equation}
\label{istar}\|i^*(v)\|\le c |v|_\frac{2N}{N+2}\ \hbox{for any}\ v\in L^{\frac{2^*}{2^*-1}}(\Omega).
\end{equation}
Then, problem~\eqref{Plog} can be restated as
\begin{equation} \label{prob:inverse}
\begin{cases}
u=i^*[f_\eps(u)],\\
u\in H^1_0 (\Omega).
\end{cases}
\end{equation}

Let $[0,\delta_N]$ be the largest interval in which the function $\delta\mapsto\delta^{N-2}|\ln\delta|$ is strictly increasing and, for $\delta\in(0,\delta_N)$ and $d\in(0,\infty)$ define
\begin{equation} \label{eq:delta}
\eps=\eps(d,\delta):=d\,\delta^{N-2}|\ln\delta|.
\end{equation}

For suitable $(d,\xi)\in(0,\infty)\times\Omega$ and $\eps$ small enough, we look for a solution to the problem~\eqref{prob:inverse} having the form
\begin{align*}
PU_{\delta,\xi} + \phi\qquad\text{with \,}\phi\in K_{\delta,\xi}^\perp,
\end{align*}
where $\delta$ and $\eps$ are related by~\eqref{eq:delta}.

\begin{remark}\label{eps:rate}
\emph{Let $d_0>1$, $d\in(d_0^{-1},d_0)$, $\eps\in(0,1)$, and 
\begin{align}\label{rate1}
\delta=\delta(d,\eps)=\left(d\frac{\eps}{|\ln\eps|}\right)^\frac{1}{N-2}.
\end{align}
Then,
\begin{align}\label{arate}
\delta^{N-2}|\ln\delta|=
\frac{d}{N-2}\frac{\Big|\ln\left(d\frac{\eps}{|\ln\eps|}
\right)\Big|}{|\ln\eps|}\ \eps
=\kappa_{\eps,d}\,\eps,
\end{align}
where
\begin{align*}
\kappa_{\eps,d} &:= \frac{d}{N-2}\frac{\Big|\ln\left(d\frac{\eps}{|\ln\eps|}
\right)\Big|}{|\ln\eps|}
=\frac{d}{N-2}\frac{|\ln(d)+\ln(\eps)-\ln|\ln\eps||}{|\ln\eps|}\\
&=\frac{d}{N-2}\left|1-\frac{\ln(d)}{|\ln\eps|}+\frac{\ln|\ln\eps|}{|\ln\eps|}
\right|=\frac{d}{N-2}(1+o(1))
\end{align*}
as $\eps\to 0.$  In particular, there is $\kappa_{\eps,d}$ bounded away from zero and infinity such that the rate~\eqref{rate1} satisfies~\eqref{arate}.}
\end{remark}

Note that $PU_{\delta,\xi} + \phi$ satisfies~\eqref{prob:inverse} if and only if the following two identities hold true:
\begin{align}
&\Pi_{\delta,\xi}^\perp\big(PU_{\delta,\xi} + \phi-i^*[f_\eps(PU_{\delta,\xi} + \phi)]\big)=0,\label{eq:perp}\\
&\Pi_{\delta,\xi}\big(PU_{\delta,\xi} + \phi-i^*[f_\eps(PU_{\delta,\xi} + \phi)]\big)=0.\label{eq:kern}
\end{align}

First, we show that, for any $(d,\xi)\in(0,\infty)\times\Omega$ and every $\eps$ small enough, there exists a unique $\phi\in K_{\delta,\xi}^\perp$ which satisfies~\eqref{eq:perp}. To this end, we consider the linear operator $L_{\delta,\xi}:K_{\delta,\xi}^\perp\to K_{\delta,\xi}^\perp$ given by
$$L_{\delta,\xi}(\phi):=\phi-\Pi_{\delta,\xi}^\perp i^*[f'_0(U_{\delta,\xi})\phi].$$

\begin{proposition} \label{prop:invertible} For any $\delta_0>0$ and for any compact subset  $D$ of $\Omega$ there exists $ C>0$ such that, for every $\xi\in D$ and $\delta\in(0,\delta_0)$,
\begin{equation}
\|L_{\delta,\xi}(\phi)\| \geq C\|\phi\| \qquad\text{for all }\phi\in K_{\delta,\xi}^\perp,
\end{equation}
and the operator $L_{\delta,\xi}:K_{\delta,\xi}^\perp\to K_{\delta,\xi}^\perp$ is invertible.
\end{proposition}

\begin{proof}
For sake of completeness, we give a   sketch of the proof which can also be found in~\cite[Lemma 1.7]{mp_2002}.
We argue by contradiction and suppose there exist sequences $\delta_n\to0$, $\xi_n\to\xi\in\Omega$, and $\phi_n,z_n\in K_{\delta_n,\xi_n}^\perp$
such that $\|\phi_n\|=1$, $\|z_n\|\to0$, and $z_n=L_{\delta_n,\xi_n}(\phi_n).$ In particular, there exists $w_n\in K_{\delta_n,\xi_n}$ such that
\begin{equation}\label{l1}\int\limits_\Omega \nabla\phi_n\nabla \varphi= \int\limits_\Omega f'_0( U_{\delta_n,\xi_n})\phi_n\varphi+\int\limits_\Omega \nabla(z_n+w_n)\nabla \varphi \ \hbox{for any}\ \varphi\in H^1_0(\Omega).\end{equation}
First of all, we claim that
$\|w_n\|\to0.$ 
Let, $w_n=\sum_{j=0}^Nc_n^j P\psi^j_{\delta_n,\xi_n}.$ By~\eqref{l3},
$\|w_n\|=\sum_{j=0}^N|c_n^j|(1+o(1))$
and, by~\eqref{l1},  
\begin{align*}\|w_n\|^2&=\underbrace{\langle \phi_n-z_n,w_n\rangle}_{=0}-\int\limits_\Omega f'_0(U_{\delta_n,\xi_n})\phi_nw_n\\
&= -\sum_{j=0}^Nc_n^j\underbrace{\int\limits_\Omega f'_0(U_{\delta_n,\xi_n}) \phi_n \psi^j_{\delta_n,\xi_n}}_{=\langle \phi_n, P\psi^j_{\delta_n,\xi_n}\rangle=0}-\sum_{j=0}^Nc_n^j\int\limits_\Omega f'_0(U_{\delta_n,\xi_n}) \phi_n\left(P\psi^j_{\delta_n,\xi_n}-\psi^j_{\delta_n,\xi_n}\right)\\
&\le \sum_{j=0}^N|c_n^j| 
\left |f'_0 (U_{\delta_n,\xi_n})\right |_{\frac N2}\underbrace{\left |P\psi^j_{\delta_n,\xi_n}-\psi^j_{\delta_n,\xi_n}\right |_{  {2N\over N-2}}}_{=o(1)} |\phi_n |
_{  {2N\over N-2}} \\
&=o(\|w_n\|)
\end{align*}
and the claim follows.

Now, we set $\tilde h(y):=\delta_n^{N-2\over2}h(\delta_n y+\xi_n),$ for $y\in \Omega_n:={\Omega-\xi_n\over\delta_n},$ so $ |\nabla \tilde h |_{2}= | \nabla  h |_{2}$ and $ | \tilde h|_{{2N\over N-2}}=| h |_{{2N\over N-2}}.$  
Then, by~\eqref{l1}, 
\begin{equation}\label{l4}\int\limits_{\Omega_n}\nabla \tilde \phi_n \nabla \tilde\varphi=\int\limits_{\Omega_n} f'_0( U)\tilde \phi_n\tilde\varphi+\int\limits_{\Omega_n}\nabla (\tilde z_n +\tilde w_n)\nabla \tilde\varphi\ \hbox{for any}\ \tilde \varphi\in C^\infty_0(\mathbb R^N).\end{equation}
Now, up to a subsequence $\tilde\phi_n\to \tilde\phi$ weakly in $\mathcal D^{1,2}(\mathbb R^N)$,
$\tilde z_n,\tilde w_n\to 0$ strongly in $\mathcal D^{1,2}(\mathbb R^N)$, and from~\eqref{l4} we get that $\phi\in\mathcal D^{1,2}(\mathbb R^N)$ solves
$$-\Delta \tilde\phi=f'_0(U)\tilde\phi \ \hbox{in}\ \mathbb R^N.$$
Moreover, since for any $j$
$$0=\int\limits_{\Omega}\nabla \tilde  \phi_n \nabla P\psi^j_{\delta_n,\xi_n}=\int\limits_{\Omega}f'_0(U_{\delta_n,\xi_n})  \tilde\phi_n  \psi^j_{\delta_n,\xi_n}=\int\limits_{\Omega_n}f'(U )  \tilde\phi_n  \psi^j\to\int\limits_{\mathbb R^N}f'(U ) \tilde\phi  \psi^j , $$
we get $\tilde\phi=0.$\\
On the other hand, testing  ~\eqref{l1} by $\phi_n$ and scaling, we have
$$1=\int\limits_{\Omega_n} f'_0( U)\tilde \phi_n^2+\int\limits_{\Omega_n}\nabla (\tilde z_n +\tilde w_n)\nabla \tilde\phi_n=o(1),$$
and a contradiction arises.

The invertibility follows from Fredholm's theory because $L_{\delta,\xi}$ is a compact perturbation of the identity.
\end{proof}

It is useful to recall the following well-known estimates.

\begin{lemma}\label{lem:U:q}
\begin{equation}\label{U:q}
\io U_{\delta,\xi}^{q}(x)\, \d x= 
\begin{cases}
\mathcal O\left( \delta^{\frac{N-2}{2}\, q}\right)\ \ &\hbox{if}\ 0<q<\frac{N}{N-2},\\
\mathcal O\left( \delta^{\frac{N}{2}}\,|\ln\delta|\right)\ &\hbox{if}\ q=\frac{N}{N-2}, \\
\mathcal O\left( \delta^{N-\frac{N-2}{2}\, q}\right)\ &\hbox{if}\ \frac{N}{N-2}<q\le 2^*,\\
\end{cases}
\end{equation}
\begin{equation}\label{psi:0:q}
\io |\psi_{\delta,\xi}^0(x)|^{q}\, \d x= 
\begin{cases}
\mathcal O\left( \delta^{\frac{N-2}{2}\, q}\right)\ \ &\hbox{if}\ 0<q<\frac{N}{N-2},\\
\mathcal O\left( \delta^{\frac{N}{2}}\,|\ln\delta|\right)\ &\hbox{if}\ q=\frac{N}{N-2}, \\
\mathcal O\left( \delta^{N-\frac{N-2}{2}\, q}\right)\ &\hbox{if}\ \frac{N}{N-2}<q\le 2^*,\\
\end{cases}
\end{equation}
and
\begin{equation}\label{psi:j:q}
\io |\psi_{\delta,\xi}^j(x)|^{q}\, \d x= 
\begin{cases}
\mathcal O\left( \delta^{\frac{N}{2}\, q}\right)\ \ &\hbox{if}\ 0<q<\frac{N}{N-1},\\
\mathcal O\Big( \delta^{\frac{N^2}{2(N-1)}}\,|\ln\delta|\Big)\ &\hbox{if}\ q=\frac{N}{N-1}, \\
\mathcal O\left( \delta^{N-\frac{N-2}{2}\, q}\right)\ &\hbox{if}\ \frac{N}{N-1}<q\le 2^*,\\
\end{cases}
\end{equation}
for $j=1,\cdots, N$.
\end{lemma}
\begin{proof} We prove the estimate~\eqref{U:q}. The other two are obtained in a similar way. In the following $C>0$ denotes a constant independent of $\delta$ and $\xi$, not necessarily the same one.  We perform the change of variable $x-\xi=\delta y$ and set $\Omega_\delta:=\frac{1}{\delta}(\Omega-\xi)$. 
By~\eqref{bub}, for $\delta\in(0,\de_0)$ with $\de_0$ small enough, 
we obtain
\begin{equation*}
\io U_{\delta,\xi}^{q}(x)\d x =\delta^{N-\frac{N-2}{2}q}\int_{\Omega_\delta} U^{q}(y)\d y.
\end{equation*}
Assume now $0<q<\frac{N}{N-2}$, since $1+r^2\ge \max\{1,r^2\}$, then
\begin{align*}
\int_{\Omega_\delta} U^{q}(y)\d y
&\le C \int_0^{c/\delta}\frac{r^{N-1}}{(1+r^2)^{\frac{N-2}{2}q}}\,d r\le C \left(\int_0^{1}r^{N-1}\,d r
+\int_1^{c/\delta} r^{N-1-(N-2)q}\,d r\right)\\
&\le C\delta^{-N+(N-2)q}.
\end{align*}
On the other hand, if $q=\frac{N}{N-2}$, then
\begin{align*}
\int_{\Omega_\delta} U^{q}(y)\d y
\le C \int_0^{c/\delta}\frac{r^{N-1}}{(1+r^2)^{\frac{N}{2}}}\,d r\le C \left(\int_0^{1}r^{N-1}\,d r
+\int_1^{c/\delta} r^{-1}\,d r\right)\le C|\ln\delta|.
\end{align*}
Finally, if $\frac{N}{N-2}<q\le 2^*$, then
\begin{align*}
\int_{\Omega_\delta} U^{q}(y)\d y
\le C \int_0^\infty \frac{r^{N-1}}{(1+r^2)^{\frac{N-2}{2}q}}\,d r = C.
\end{align*}
This ends the proof.
\end{proof}	

\begin{lemma}\label{k1}
\begin{equation}\label{s1}
\left |f_0\left(PU_{\delta,\xi}\right)-f_0\left(U_{\delta,\xi}\right)\right |_{{2N\over N+2}}=
\begin{cases}
\mathcal O\left( \delta^{N-2}\right)\ &\hbox{if}\ 3\le N\le 5,\\
\mathcal O\left( \delta^4|\ln\delta|^{2/3}\right)\ &\hbox{if}\ N=6, \\
\mathcal O\left( \delta^{N+2\over2}\right)\ &\hbox{if}\ N\ge 7,\\
\end{cases}
\end{equation}
\begin{equation}\label{s2}
\left |f'_0\left(PU_{\delta,\xi}\right)-f'_0\left(U_{\delta,\xi}\right)\right |_{{N\over 2}}
=
\begin{cases}
\mathcal O\left( \delta\right)\ &\hbox{if}\ N=3,\\
\mathcal O\left( \delta^{2}|\ln\delta|^{1/2}\right)\ &\hbox{if}\ N=4, \\
\mathcal O\left( \delta^{2}\right)\ &\hbox{if}\ N\ge 5,\\
\end{cases}
\end{equation}
and
\begin{align}\label{f0:f'0}
\big |f_0(PU_{\delta,\xi})&-f_0(U_{\delta,\xi})-f'_0(U_{\delta,\xi}) ( PU_{\delta,\xi}(x) -U_{\delta,\xi}(x))\big |_{{\frac{N}2}}
=
\begin{cases}
\mathcal O\left( \delta^{\frac{N+2}{2}}\right)\ &\hbox{if}\ N=3,\\
\mathcal O\left( \delta^{\frac{N+2}{2}}|\ln\delta|^{1/2}\right)\ &\hbox{if}\ N=4, \\
\mathcal O\left( \delta^{\frac{N+2}{2}}\right)\ &\hbox{if}\ N\ge 5,\\
\end{cases}
\end{align}
\end{lemma}
\begin{proof} 
The following inequalities are well known. For any $a>0$ and $b\in\mathbb R$,
\begin{equation}\label{y1}
\left||a+b|^q-a^q\right|\le 
\begin{cases}
c(q)\min\{|b|^q,a^{q-1}|b|\} \ &\hbox{if}\ 0<q<1,\\
c(q)\left(|b|^q+a^{q-1}|b|\right) \ &\hbox{if}\ q\ge1,\\
\end{cases}
\end{equation}
and
\begin{equation}\label{y2}
\left||a+b|^q(a+b)-a^{q+1}-(1+q)a^qb\right|\le 
\begin{cases}
c(q)\min\{|b|^{q+1},a^{q-1}b^2\}\ &\hbox{if}\ 0<q<1,\\
c(q)\left(|b|^{q+1}+a^{q-1}b^2\right)\ &\hbox{if}\ q\ge1.\\
\end{cases}
\end{equation}
Estimates \eqref{s1}, \eqref{s2}, and \eqref{f0:f'0} follow from these inequalities and Lemma \ref{lem:U:q}. 
\end{proof}

\begin{lemma}\label{k2}
\begin{equation}\label{s11}
\left |f_\eps\left(PU_{\delta,\xi}\right)-f_0\left(PU_{\delta,\xi}\right)\right|_{{2N\over N+2}}= \mathcal O\left(\eps\ln |\ln\delta|\right)\end{equation}
and
\begin{equation}\label{s12}
\left|f'_\eps\left(PU_{\delta,\xi}\right)-f'_0\left(PU_{\delta,\xi}\right)\right|_{{N\over 2}}= \mathcal O\left(\eps\ln |\ln\delta|\right).
\end{equation}
\end{lemma}
\begin{proof}
In the following $C>0$ denotes a positive constant, independent of $\delta$, $\eps$, and $\xi\in(0,1)$, not necessarily the same one. We show first ~\eqref{s11}. By 
Lemma~\ref{lem:meanvalue}  and the maximum principle,
\begin{align*}
|f_\eps(PU_{\delta,\xi})-f_0(PU_{\delta,\xi})| &\leq\eps (PU_{\delta,\xi})^{2^*-1}\ln\ln(\e+PU_{\delta,\xi}) \\
&\leq \eps\, U_{\delta,\xi}^{2^*-1}\ln\ln(\e+U_{\delta,\xi}).
\end{align*}
Next, we scale $x-\xi=\delta y,$  $y\in \Omega_\delta:=\frac{1}{\delta}(\Omega-\xi)$ and we get, for $\delta\in(0,1)$,
\begin{align*}
&\left(\io \left|U_{\delta,\xi}^{2^*-1}(x)\ln\ln(\e+U_{\delta,\xi}(x))\right|^\frac{2^*}{2^*-1}\d x\right)^\frac{2^*-1}{2^*} \\
&\qquad \leq\left(\int_{\Omega_\delta} U^{2^*}(y)\left|\ln\ln\left(\e+ \delta^{-\frac{N-2}{2}} U(y)\right)\right|^\frac{2^*}{2^*-1}\d y\right)^\frac{2^*-1}{2^*}\\ 
&\qquad\leq C\,
\left|\ln\ln\left(\e+ \delta^{-\frac{N-2}{2}}\al_N \right)\right| \le C \ln|\ln\delta|
\end{align*}
and~\eqref{s11} follows.

Now we show~\eqref{s12}. By Lemma~\ref{lem:meanvalue}, for $\eps$ small enough we have that
$$|f'_{\eps}(u)-f'_0(u)| \leq  C \eps |u|^{2^*-2}\big( \ln \ln(\e+|u|) + 1\big).$$
Next, we scale $x-\xi=\delta y,$  $y\in \Omega_\delta:=\frac{1}{\delta}(\Omega-\xi)$ and then, for $\delta\in(0,\frac{1}{2})$,
\begin{align*}
&\left(\io |U_{\delta,\xi}^{2^*-2}\left(\ln \ln(\e+U_{\delta,\xi}) + 1\right) |^\frac{2^*}{2^*-2}\right)^\frac{2^*-2}{2^*}\\
&\qquad=\left(\int_{\Omega_\delta} U^{2^*}(y)\left(\ln \ln(\e+\delta^{-\frac{N-2}{2}}U(y)) + 1\right)^\frac{2^*}{2^*-2}\d y\right)^\frac{2^*-2}{2^*} \\
&\qquad\leq C
\left(\ln \ln(\e+\delta^{-\frac{N-2}{2}}\al_N) + 1\right)   \leq C
\ln|\ln\delta|
\end{align*}
and~\eqref{s12} follows.

\end{proof}

\begin{proposition}\label{prop:fixed_point}

For any compact subset $X$ of \,$(0,\infty)\times\Omega$ there is $\delta_0(X)=\delta_0>0$ such that, for every $(d,\xi)\in X$ and $\delta\in(0,\delta_0)$, there exists a unique $\phi=\phi_{\delta,\xi}\in K_{\delta,\xi}^\perp$ which solves equation~\eqref{eq:perp} with $\eps=d\,\delta^{N-2}|\ln\delta|$ and satisfies
\begin{equation}\label{phibd}
\|\phi_{\delta,\xi}\| =
\begin{cases}
\mathcal O\left(\delta^{N-2}\,|\ln\delta|\, (\ln|\ln\delta|) \right)\ &\hbox{if}\ 3\le N\le 6,\\
\mathcal O\left(  \delta^{N+2\over2}\right)\ &\hbox{if}\ N\ge 7.
\end{cases}
\end{equation}
\end{proposition}

\begin{proof}

Let $(d,\xi)\in X$, $\delta\in(0,1)$. Note that $\phi\in K_{\delta,\xi}^\perp$ solves~\eqref{eq:perp} if and only if $\phi$ is a fixed point of the  operator  $T_{\delta,\xi}:K_{\delta,\xi}^\perp\to K_{\delta,\xi}^\perp$ defined by
\begin{align*}
T_{\delta,\xi}(\phi):= L_{\delta,\xi}^{-1}\,\Pi_{\delta,\xi}^\perp\, i^*  \Big\{ &\left[f_\eps(PU_{\delta,\xi}+ \phi)-f_\eps(PU_{\delta,\xi})-f'_\eps(PU_{\delta,\xi})\phi\right]\\
&+
\left[ f'_\eps(PU_{\delta,\xi})-f'_0(PU_{\delta,\xi})\right]\phi+\left[ f'_0(PU_{\delta,\xi})-f'_0(U_{\delta,\xi})\right]\phi\\
&+
\left[ f_\eps(PU_{\delta,\xi})-f_0(PU_{\delta,\xi})\right] +\left[ f_0(PU_{\delta,\xi})-f_0(U_{\delta,\xi})\right] \ \Big\}.
\end{align*}

We shall prove that $T_{\delta,\xi}$ is a contraction in a suitable ball.  
Hereafter $C>0$ denotes a positive constant, independent of $(d,\xi)\in X$ and $\delta\in(0,1)$, not necessarily the same one. Proposition~\ref{prop:invertible} and Sobolev's inequality yield
\begin{align*}
\|T_{\delta,\xi}(\phi)\|  &\leq   C\big|f_\eps(PU_{\delta,\xi}+ \phi) -f_\eps(PU_{\delta,\xi}) - f'_\eps(PU_{\delta,\xi})\phi\big|_\frac{2^*}{2^*-1} \\
&\qquad +C \big|\big(f'_\eps(PU_{\delta,\xi})-f'_0(PU_{\delta,\xi})\big)
\phi\big|_\frac{2^*}{2^*-1}+ C\big|\big(f'_0(PU_{\delta,\xi})-f'_0(U_{\delta,\xi})\big)
\phi\big|_\frac{2^*}{2^*-1}\\
&\qquad   + C|f_\eps(PU_{\delta,\xi}) - f_0(PU_{\delta,\xi}) |_\frac{2^*}{2^*-1} +C|f_0(PU_{\delta,\xi}) - f_0(U_{\delta,\xi}) |_\frac{2^*}{2^*-1} \\
&=:A_1+A_2+A_3+A_4+A_5.
\end{align*}
By~\eqref{s1} and~\eqref{s11},
$$ 
|A_4|_{{2N\over N+2}}
+|A_5|_{{2N\over N+2}}\le
\mathcal O(R_\delta),$$
where 
\begin{align*}
R_\delta=
\begin{cases} 
\delta^{N-2}\,|\ln\delta|\, \ln|\ln\delta|  \ &\hbox{if}\ 3\le N\le 6,\\
\delta^{N+2\over2} \ &\hbox{if}\ N\ge 7.\end{cases}
\end{align*}
Next, we estimate the other terms.

$A_1)$ By the mean value theorem, there exists $\theta=\theta(x)\in(0,1)$ such that
\begin{align}
A_1&=\big|f_\eps(PU_{\delta,\xi} + \phi) - f_\eps(PU_{\delta,\xi}) - f'_\eps(PU_{\delta,\xi})\phi\big|_\frac{2^*}{2^*-1} \label{A1}\\
&=\big|\big(f'_\eps(PU_{\delta,\xi} + \theta\phi) -  f'_\eps(PU_{\delta,\xi})\big)\phi\big|_\frac{2^*}{2^*-1}.\notag
\end{align}
If $N< 6$, from Lemma~\ref{lem:meanvalue} and Hölder's inequality, 
\begin{align*}
A_1&\leq  C\,\big(\big| |\phi|^{2^*-1}\big|_\frac{2^*}{2^*-1}+\big|U_{\delta,\xi}^{2^*-3}\phi^2\big|_\frac{2^*}{2^*-1}
\big) 
= C\left[|\phi|_{2^*}^{2^*-1}+\left(\io \left(U_{\delta,\xi}^{2^*-3}\phi^2\right)^\frac{2^*}{2^*-1}\right)^\frac{2^*-1}{2^*} \right]\\
&\leq C\left(|\phi|_{2^*}^{2^*-1}+
|U_{\delta,\xi}|_{2^*}^{2^*-3}|\phi|_{2^*}^2  \right),
\end{align*}
while if $N=6$,
\begin{align*}
A_1&\leq  C\,\big(\big| |\phi|^{2^*-1}\big|_\frac{2^*}{2^*-1}+\big|\phi^2\big|_\frac{2^*}{2^*-1}\big) = C\left[|\phi|_{2^*}^{2^*-1}+\left(\io |\phi|^{2^*} \right)^\frac{2^*-1}{2^*} \right]= 2C\,|\phi|_{2^*}^{2^*-1}\, .
\end{align*}
On the other hand, if $N>6$, we obtain         
\begin{align*}
A_1&\leq  C\,\big(\big| |\phi|^{2^*-1}\big|_\frac{2^*}{2^*-1}
+\eps \, \big| U_{\delta,\xi}^{2^*-2}\phi\big|_\frac{2^*}{2^*-1}
\big) = C\left[|\phi|_{2^*}^{2^*-1} +\eps
\left(\io \left(U_{\delta,\xi}^{2^*-2}|\phi|\right)^\frac{2^*}{2^*-1}\right)^\frac{2^*-1}{2^*}\right]\\
&\leq C\left(|\phi|_{2^*}^{2^*-1}+
\eps |U_{\delta,\xi}|_{2^*}^{2^*-2}|\phi|_{2^*} \right).
\end{align*}
Now, Sobolev's inequality gives
\begin{equation}\label{A1:2}
A_1\leq
\begin{cases}
C\,\big(1+\|\phi\|^{2^*-3}  \big)\,\|\phi\|^2\ &\text{if}\ 3\le N\le 5,\\
C\,\|\phi\|^2 \ &\text{if}\  N=6,\\
C\,\big(\eps +\|\phi\|^{2^*-2} \big)\,\|\phi\|\ &\text{if}\  N\ge 7.
\end{cases}
\end{equation}

\medskip

$A_2)$ By Holder's inequality and ~\eqref{s12},
$$\big|\big(f'_\eps(PU_{\delta,\xi})-f'_0(PU_{\delta,\xi})\big)
\phi\big|_\frac{2^*}{2^*-1}\le \big|f'_\eps(PU_{\delta,\xi})-f'_0(PU_{\delta,\xi})
\big|_\frac{N}{2 } |\phi|_{2^*}\le C\eps\ln|\ln\delta|\ \|\phi\|.$$

\medskip

$A_3)$ By Holder's inequality and ~\eqref{s2},
\begin{align*}
\big|\big(f'_0(PU_{\delta,\xi})-f'_0(U_{\delta,\xi})\big)
\phi\big|_\frac{2^*}{2^*-1}\le \big|f'_0(PU_{\delta,\xi})-f'_0(U_{\delta,\xi})
\big|_\frac{N}{2 }|\phi|_{2^*}\le 
\begin{cases}
C\delta\, \|\phi\|  &\hbox{if}\ N=3,\\
C\delta^2|\ln\delta|^{1/2}\|\phi\| &\hbox{if}\ N=4, \\
C\delta^{2}\|\phi\| &\hbox{if}\ N\ge 5.\\
\end{cases}
\end{align*}

\medskip

Collecting all the previous estimates, we deduce that there exist  $R^*>0$ and $\delta_0>0$ such that   for any $\delta\in(0,\delta_0)$ 
\begin{equation*}
\|T_{\delta,\xi}(\phi)\| \leq      R^* R_\delta\ \hbox{for any}\ \phi\in B_\delta:=\{\phi\in K_{\delta,\xi}^\perp:\|\phi\|\leq   R ^*  R_\delta   \}. 
\end{equation*}

Next, we show that $T_{\delta,\xi}$ is a contraction. To this end, let $\phi_1,\phi_2\in B_\delta$.
We have
\begin{align*}
\|T_{\delta,\xi}(\phi_1)-T_{\delta,\xi}(\phi_2)\| \le
& C\, |f_\eps(PU_{\delta,\xi}+ \phi_1)-f_\eps(PU_{\delta,\xi}+ \phi_2)-f'_\eps(PU_{\delta,\xi})(\phi_1-\phi_2)|_\frac{2^*}{2^*-1}\\
&+C
\left|[ f'_\eps(PU_{\delta,\xi})-f'_0(PU_{\delta,\xi})](\phi_1-\phi_2)\right|_\frac{2^*}{2^*-1}\\
&+C\left|[ f'_0(PU_{\delta,\xi})-f'_0(U_{\delta,\xi})](\phi_1-\phi_2)\right|_\frac{2^*}{2^*-1}=:a_1+a_2+a_3.
\end{align*}
To estimate $a_1$, $a_2$, and $a_3$ we argue as we did above for $A_1$, $A_2$, and $A_3$.

$a_1$) By the mean value theorem, there exists  $\theta=\theta(x)\in(0,1)$ such that
\begin{align*}
a_1&=|f_\eps(PU_{\delta,\xi}+ \phi_1)-f_\eps(PU_{\delta,\xi}+ \phi_2)-f'_\eps(PU_{\delta,\xi})(\phi_1-\phi_2)|_\frac{2^*}{2^*-1}\\
& = \big|\big(f'_\eps(PU_{\delta,\xi} + \phi_\theta)  -f'_\eps(PU_{\delta,\xi}) \big) (\phi_2-\phi_1)\big|_\frac{2^*}{2^*-1},
\end{align*}
where $\phi_\theta:=(1-\theta)\phi_1+\theta\phi_2$.

If $N< 6$, from Lemma~\ref{lem:meanvalue} and Hölder's inequality we get 
\begin{align*}
a_1&\leq  C\,\Big(\big| |\phi_\theta|^{2^*-2} (\phi_2-\phi_1)\big|_\frac{2^*}{2^*-1}+\big|U_{\delta,\xi}^{2^*-3}\phi_\theta(\phi_2-\phi_1)\big|_\frac{2^*}{2^*-1}\Big) \\
&\le C\left[|\phi_\theta|_{2^*}^{2^*-2}|\phi_2-\phi_1|_{2^*}+\left(\io \left(U_{\delta,\xi}^{2^*-3}\phi_\theta(\phi_2-\phi_1)\right)^\frac{2^*}{2^*-1}\right)^\frac{2^*-1}{2^*} \right]\\
&\leq C\left(|\phi_\theta|_{2^*}^{2^*-2}+|U_{\delta,\xi}|_{2^*}^{2^*-3}|\phi_\theta|_{2^*}  \right)\,|\phi_2-\phi_1|_{2^*},
\end{align*}
while, if $N= 6$, we obtain 
\begin{align*}
a_1&\leq  C\, |\phi_\theta|_{2^*}\,|\phi_2-\phi_1|_{2^*}  .
\end{align*}
On the other hand, if $N>6$,
\begin{align*}
a_1&\leq  C\,\big(\big| |\phi_\theta|^{2^*-2}(\phi_2-\phi_1)\big|_\frac{2^*}{2^*-1}
+\eps \, \big| U_{\delta,\xi}^{2^*-2}(\phi_2-\phi_1)\big|_\frac{2^*}{2^*-1}
\big) \\
&\le C\left[|\phi_\theta|_{2^*}^{2^*-2}\,|\phi_2-\phi_1|_{2^*} +\eps
\left(\io \left(U_{\delta,\xi}^{2^*-2}|\phi_2-\phi_1|\right)^\frac{2^*}{2^*-1}\right)^\frac{2^*-1}{2^*}\right]\\
&\leq C\left(|\phi_\theta|_{2^*}^{2^*-2}+
\eps  \right)\,|\phi_2-\phi_1|_{2^*}.
\end{align*}
Now, Sobolev's inequality gives
\begin{equation*}
a_1\leq
C\,\Big(\|\phi_\theta\|^{2^*-2} + \max\{\|\phi_\theta\|,\eps\} \Big)\,\|\phi_2-\phi_1\|.
\end{equation*}

$a_2$) By Holder's inequality and ~\eqref{s12},
\begin{align*}
a_2&=\big|\big(f'_\eps(PU_{\delta,\xi})-f'_0(PU_{\delta,\xi})\big)
(\phi_2-\phi_1)\big|_\frac{2^*}{2^*-1}\\
&\le \big|f'_\eps(PU_{\delta,\xi})-f'_0(PU_{\delta,\xi})
\big|_\frac{N}{2 }|\phi_2-\phi_1|_{2^*}\le C\eps\ln|\ln\delta|\|\phi_2-\phi_1\|.
\end{align*}

$a_3$) By Holder's inequality and ~\eqref{s2},
\begin{align*}
a_3&=\big|\big(f'_0(PU_{\delta,\xi})-f'_0(U_{\delta,\xi})\big)
(\phi_2-\phi_1)\big|_\frac{2^*}{2^*-1}\\
&\le \big|f'_0(PU_{\delta,\xi})-f'_0(U_{\delta,\xi})
\big|_\frac{N}{2 }|\phi_2-\phi_1|_{2^*}\le 
\begin{cases}
C\delta\, \|\phi_2-\phi_1\|  &\hbox{if}\ N=3,\\
C\delta^2|\ln\delta|^{1/2}\|\phi_2-\phi_1\| &\hbox{if}\ N=4, \\
C\delta^{2}\|\phi_2-\phi_1\| &\hbox{if}\ N\ge 5.\\
\end{cases}
\end{align*}

From the above estimates we conclude that
\begin{align*}
\|T_{\delta,\xi}(\phi_2)-T_{\delta,\xi}(\phi_1)\|
&\leq
C\Big(R^* R_\delta +(R^* R_\delta)^{2^*-2} +\eps\ln|\ln\delta|+\de\Big)\|\phi_2-\phi_1\|. 
\end{align*}
Hence, there exists $\delta_0\in(0,1)$ such that $T_{\delta,\xi}:B_\delta\to B_\delta$ is a contraction for all $\delta\in(0,\delta_0)$. By Banach's fixed point theorem, $T_{\delta,\xi}:B_\delta\to B_\delta$ has a unique fixed point, as claimed.
\end{proof}

\section{The finite dimensional problem}\label{sec:fdp}

In the previous section we proved that,  if $\eps=d\, {\delta^{N-2}|\ln\delta|}$ is small enough, then, for each $d>0$   and $\xi\in\Omega$, there exists a unique $\phi=\phi_{\delta,\xi}\in K_{\delta,\xi}^\perp$ which solves equation~\eqref{eq:perp}, \textit{i.e.},
$$PU_{\delta,\xi} + \phi-i^*[f_\eps(PU_{\delta,\xi} + \phi)]\in K_{\delta,\xi}.$$
Hence, there exist $c_{\delta,\xi}^0,c_{\delta,\xi}^1,\ldots,c_{\delta,\xi}^N\in\r$ such that
\begin{equation*}
PU_{\delta,\xi} + \phi-i^*[f_\eps(PU_{\delta,\xi} + \phi)]=\sum_{i=0}^Nc_{\delta,\xi}^iP\psi^{i}_{\delta,\xi}.
\end{equation*}
In order to show that   $PU_{\delta,\xi} + \phi$ solves~\eqref{prob:inverse}, we need to prove that it solves~\eqref{eq:kern}. That is, we need to show that there exists $d_\eps>0$ and $\xi_\eps\in \Omega$ such that the $c_{\delta_\eps,\xi_\eps}^i$'s are zero  for $\eps$ small enough.

\begin{proposition}\label{main:prop} Let $\xi_0\in\Omega$ be a non-degenerate critical point of the Robin function $\varrho_\Omega$. Then there exist  $\xi_\eps\to\xi_0$ and $\delta_\eps\to 0$ given by $\eps= d_\eps \delta_\eps^{N-2}|\ln\delta_\eps|$ with $d_\eps\to d_0>0$ and
such that 
\begin{align*}
PU_{\delta_\eps,\xi_\eps} + \phi-i^*[f_\eps(PU_{\delta_\eps,\xi_\eps} + \phi)]=0,
\end{align*}
where $\phi=\phi_{\delta_\eps,\xi_\eps}$ is given by Proposition~\ref{prop:fixed_point}.
\end{proposition}
\begin{proof} We split the proof in two steps.

{\em Step 1}. We take the inner product of~\eqref{eq:c_i} with $P\psi^j_{\delta,\xi}$ and compute each side of the identity
\begin{equation} \label{eq:c_i}
\langle PU_{\delta,\xi} + \phi-i^*[f_\eps(PU_{\delta,\xi} + \phi)],P\psi^j_{\delta,\xi} \rangle =\sum_{i=0}^Nc_{\delta,\xi}^i\langle P\psi^{i}_{\delta,\xi},P\psi^j_{\delta,\xi} \rangle.
\end{equation}

The left-hand side is
\begin{align*}
&LHS:=\langle PU_{\delta,\xi} + \phi-i^*[f_\eps(PU_{\delta,\xi} + \phi)],P\psi^j_{\delta,\xi} \rangle  \\
&=\langle PU_{\delta,\xi},P\psi^j_{\delta,\xi} \rangle - \io f_\eps(PU_{\delta,\xi} + \phi)\,P\psi^j_{\delta,\xi}= \io f_0(U_{\delta,\xi})\,P\psi^j_{\delta,\xi} -\io f_\eps(PU_{\delta,\xi} + \phi)\,P\psi^j_{\delta,\xi} \\
&=\io \left(f_0(U_{\delta,\xi})-f_0(PU_{\delta,\xi})\right)\, \psi^j_{\delta,\xi}
+\io \left(f_0(U_{\delta,\xi})-f_0(PU_{\delta,\xi})\right)\,(P\psi^j_{\delta,\xi}-\psi^j_{\delta,\xi})\\
&+ \io \left(f_0(PU_{\delta,\xi})-f_\eps(PU_{\delta,\xi})\right)\,\psi^j_{\delta,\xi}
+ \io \left(f_0(PU_{\delta,\xi})-f_\eps(PU_{\delta,\xi})\right)\,(P\psi^j_{\delta,\xi}-\psi^j_{\delta,\xi})\\
&-\io \left(f_\eps(PU_{\delta,\xi} + \phi)-f_\eps(PU_{\delta,\xi} )-f'_\eps(PU_{\delta,\xi} )\phi\right)\,P\psi^j_{\delta,\xi}
-\io \left(f'_\eps(PU_{\delta,\xi} )-f'_0(PU_{\delta,\xi} )\right)\phi\,P\psi^j_{\delta,\xi}\\
&-\io \left(f'_0(PU_{\delta,\xi} )-f'_0(U_{\delta,\xi} )\right)\phi \,P\psi^j_{\delta,\xi}-\io  f'_0(U_{\delta,\xi} )\phi \,(P\psi^j_{\delta,\xi}-\psi^j_{\delta,\xi})-\underbrace{\io  f'_0(U_{\delta,\xi} )\phi\,\psi^j_{\delta,\xi}}_{=0}\\
&=:I_1+I_2+I_3+I_4+I_5+I_6+I_7+I_8.
\end{align*}
Next, we estimate each summand. The leading terms are $I_1$  and  $I_3$.

\begin{itemize}
\item[($I_1$)]  
We have
\begin{align*}
&\io  \left(f_0(U_{\delta,\xi})- f_0(PU_{\delta,\xi})\right)  \psi^j_{\delta,\xi}\\ 
&\qquad=
-\io f'_0(U_{\delta,\xi})( PU_{\delta,\xi}(x)-U_{\delta,\xi}(x)) \psi^j_{\delta,\xi}
\\
&\qquad\qquad-\io \left(f_0(PU_{\delta,\xi})- f_0(U_{\delta,\xi})-f'_0(U_{\delta,\xi})( PU_{\delta,\xi}(x)-U_{\delta,\xi}(x))\right) \psi^j_{\delta,\xi}.
\end{align*}

By~\eqref{psi:0:q},~\eqref{psi:j:q}, and~\eqref{f0:f'0},
\begin{align*} 
&\io \Big(f_0(PU_{\delta,\xi})- f_0(U_{\delta,\xi})-f'_0(U_{\delta,\xi})( PU_{\delta,\xi}(x) -U_{\delta,\xi}(x))\Big)\psi^j_{\delta,\xi} \\
&\qquad \le
\left|\Big(f_0(PU_{\delta,\xi})- f_0(U_{\delta,\xi})-f'_0(U_{\delta,\xi})( PU_{\delta,\xi}(x)-U_{\delta,\xi}(x))\Big)\right|_{{\frac{N}2}}
\left|\psi^j_{\delta,\xi}\right|_{{\frac{N}{N-2} }}
=o(\delta^{N-1}) .
\end{align*}

Moreover,
\begin{align*}
&-\io f'_0(U_{\delta,\xi})( PU_{\delta,\xi}(x)-U_{\delta,\xi}(x)) \psi^j_{\delta,\xi}\\ 
&\qquad =p\io U^{p-1}_{\delta,\xi}(x) \delta^\frac{N-2}{2}\big(  \alpha_NH(x,\xi)+\mathcal O(\delta)\big)
\psi^j_{\delta,\xi}(x)\, dx.\\
\end{align*}

If $j=0$, we scale $x=\xi + \delta y$ to obtain
\begin{align*}
&p\,\alpha_N\,\delta^\frac{N-2}{2}\, \io U^{p-1}_{\delta,\xi}(x)  H(x,\xi) \psi^0_{\delta,\xi}(x) dx\\
&\qquad =p\,\alpha_N\, \delta^\frac{N-2}{2}\, \delta^{N-\frac{N-2}{2}\frac{N+2}{N-2}}\, \int_{\Omega-\xi\over\delta} U^{p-1}(y)   H(\xi + \delta y,\xi) \psi^0(y) dy\\ 
&\qquad =\alpha_N\,   A\delta^{N-2} \big(H( \xi,\xi) +\mathcal O(\delta)\big),
\end{align*}
where
$A:=p\int_{\mathbb R^N} U^{p-1}(y) \psi^0(y)dy.$ 

If $j=1,\dots,N$, taking into account that
$\psi^j_{\delta,\xi}(x)=\delta \partial_{\xi_j} U_{\delta,\xi}(x)$ and setting $x=\xi + \delta y$, we get
\begin{align*}
&p\,\alpha_N\, \delta^\frac{N-2}{2}\, \io U^{p-1}_{\delta,\xi}(x)  H(x,\xi) \psi^j_{\delta,\xi}(x) dx=\alpha_N\,\delta^\frac{N }{2} \io \partial_{\xi_j}U^{p }_{\delta,\xi}(x)  H(x,\xi) dx\\ 
&= \alpha_N\,\delta^\frac{N }{2}\left( \partial_{\xi_j} \io U^{p }_{\delta,\xi}(x)  H(x,\xi) dx -\io U^{p }_{\delta,\xi}(x)  \partial_{\xi_j} H(x,\xi) dx\right)\\
&
= \alpha_N\, \delta^\frac{N }{2}\left( \delta^{N-2\over2}\partial_{\xi_j} \int_{\Omega-\xi\over\delta} U^{p}(y) H(\xi + \delta y,\xi)  dy-\delta^{N-2\over2}\int_{\Omega-\xi\over\delta} U^{p}(y) \partial_{\xi_j}  H(\xi + \delta y,\xi)  dy\right)\\
&=  \alpha_N B\delta^{N-1}\bigg(\underbrace{ \partial_{\xi_j}\big(H( \xi,\xi)\big)-\partial_{\xi_j}H(\xi,\xi)}_{=\frac12\partial_{\xi_j}\varrho(\xi)} 
+ \mathcal O(\delta)\bigg)
\end{align*}
where
$B:=\int\limits_{\mathbb R^N} U^{p }(y)dy.$ 
A straightforward computation shows that $A=\frac{N-2}2B$
(see also Remark B.2 in~\cite{mp_2002}).  Hence,
\begin{equation*}
-\io f'_0(U_{\delta,\xi})( PU_{\delta,\xi}(x)-U_{\delta,\xi}(x)) \psi^j_{\delta,\xi}= 
\begin{cases}
\alpha_N\, A\delta^{N-2} H( \xi,\xi)+
\mathcal O\left( \delta^{N-1}\right)\ &\hbox{if}\ j=0,\\
\frac12 \alpha_N\, B\delta^{N-1} \partial_{\xi_j}\varrho(\xi)+
\mathcal O\left( \delta^{N}\right)\ &\hbox{if}\ j=1,\dots,N.
\end{cases}
\end{equation*}
Consequently, 
\begin{equation}
I_1=
\begin{cases}
\alpha_N\, A\delta^{N-2} H( \xi,\xi)+
o\left( \delta^{N-2}\right)\ &\hbox{if}\ j=0,\\
\frac12 \alpha_N\, B\delta^{N-1} \partial_{\xi_j}\varrho(\xi)+
o\left( \delta^{N-1}\right)\ &\hbox{if}\ j=1,\dots,N.
\end{cases}
\end{equation}

\item[($I_2$)]  By~\eqref{l2} and~\eqref{s1} we deduce
\begin{align*} 
I_2&=\io \left(f_0(U_{\delta,\xi})-f_0(PU_{\delta,\xi})\right)\,(P\psi^j_{\delta,\xi}-\psi^j_{\delta,\xi})\\
&=\mathcal O\left(
\left|f_0\left(PU_{\delta,\xi}\right)-f_0\left(U_{\delta,\xi}\right)\right|_{{2N\over N+2}}
\left|P\psi^j_{\delta,\xi}-\psi^j_{\delta,\xi}\right|_{{2N\over N-2}}\right)
=
\begin{cases}
o(\delta^{N-2}) \ &\hbox{if}\ j=0,\\
o(\delta^{N-1}) \ &\hbox{if}\ j=1,\cdots,N.
\end{cases}
\end{align*}

\item[($I_3$)] The proof of this estimate is long, so we postpone the details to an appendix. If $j=0$, by Lemma~\ref{I3}, we obtain that
$$ \io \left(f_0(PU_{\delta,\xi})-f_\eps(PU_{\delta,\xi})\right)\,\psi^0_{\delta,\xi}=-\frac{2d}{N-2}\mathfrak B \frac{\eps}{|\ln\delta|}+o\left( \frac{\eps}{|\ln\delta|}\right),
$$ 
where $\mathfrak B>0$, and, if $j=1,\dots,N$, using Remark~\ref{rem:1}, we get
$$ 
\io\left(f_0(PU_{\delta,\xi})-f_\eps(PU_{\delta,\xi})\right) \,\psi^j_{\delta,\xi}= 
\begin{cases}
o(\delta^{N-2})\ &\text{if}\ 3\le N\le 4,\\
o(\delta^{N-1})\ &\text{if}\ N\ge 5.
\end{cases}
$$

\item[($I_4$)]  By~\eqref{l2} and~\eqref{s11},
\begin{align*} 
&\io\left(f_\eps(PU_{\delta,\xi})-f_0(PU_{\delta,\xi})\right)\,(P\psi^j_{\delta,\xi}-\psi^j_{\delta,\xi})\\
&=\mathcal O\left(
\left|f_\eps\left(PU_{\delta,\xi}\right)-f_0\left(PU_{\delta,\xi}\right)\right|_{{2N\over N+2}}
\left|P\psi^j_{\delta,\xi}-\psi^j_{\delta,\xi}\right|_{{2N\over N-2}}\right)\\
&=
\begin{cases}
\mathcal O\left(\eps \delta^{\frac{N-2}2}\ln|\ln\delta|\right) \ &\hbox{if}\ j=0,\\
\mathcal O\left(\eps \delta^{\frac{N}2}\ln|\ln\delta|\right)\ &\hbox{if}\ j=1,\dots,N,
\end{cases}\\
&=
\begin{cases}
o\left(\delta^{N-2}\right) \ &\hbox{if}\ j=0\ \text{and}\ 3\le N\le 4,\\
o\left(\delta^{N-1}\right)\ &\hbox{if}\ j=0\ \text{and}\ N \ge 5,\\
o\left(\delta^{N-1}\right)\ &\hbox{if}\ j=1,\dots,N,\ \text{and}\ 3\le N\le 4,\\
o\left(\delta^{N}\right)\ &\hbox{if}\ j=1,\dots,N,\ \text{and}\
\ N \ge 5.
\end{cases}
\end{align*}

\item[($I_5$)]  By \eqref{l2}, \eqref{psi:0:q}, \eqref{psi:j:q},~\eqref{A1}, and \eqref{A1:2},
\begin{align*} 
&\io \left(f_\eps(PU_{\delta,\xi} + \phi)-f_\eps(PU_{\delta,\xi} )-f'_\eps(PU_{\delta,\xi} )\phi\right)\,P\psi^j_{\delta,\xi}\\
&=\mathcal O\left(
\left|f_\eps(PU_{\delta,\xi} + \phi)-f_\eps(PU_{\delta,\xi} )-f'_\eps(PU_{\delta,\xi} )\phi\right|_{{2N\over N+2}}
\left|P\psi^j_{\delta,\xi} \right|_{{2N\over N-2}}\right)\\
&\le \left.
\begin{cases}
C\,\big(1+\|\phi\|^{2^*-3}  \big)\,\|\phi\|^2\ &\text{if}\ 3\le N\le 5,\\
C\,\|\phi\|^2 \ &\text{if}\  N=6,\\
C\,\big(\eps +\|\phi\|^{2^*-2} \big)\,\|\phi\|\ &\text{if}\  N\ge 7,
\end{cases}\right\}
=
\begin{cases}
\mathcal  O\left(\delta^{2}|\ln \de|^2\,(\ln|\ln \de|)^2\right)\ &\hbox{if}\ N=3,\\
o\left(\delta^{N-1}\right)\ &\hbox{if}\ N\ge 4,\\
\end{cases}\\
&=\begin{cases}
o\left(\delta^{N-2}\right)\ &\hbox{if}\ N=3,\\
o\left(\delta^{N-1}\right)\ &\hbox{if}\ N\ge 4.
\end{cases}  
\end{align*}

\item[($I_6$)] By~\eqref{s12} and~\eqref{phibd},
\begin{align*} 
&\io \left(f'_\eps(PU_{\delta,\xi} )-f'_0(PU_{\delta,\xi} )\right)\phi\,P\psi^j_{\delta,\xi}\\
&=\mathcal O\left(
\left|f'_\eps\left(PU_{\delta,\xi}\right)-f'_0\left(PU_{\delta,\xi}\right)\right|_{{N\over  2}}
\left|\phi\right|_{{2N\over N-2}}\left|P\psi^j_{\delta,\xi}\right|_{{2N\over N-2}}\right)\\
&=\mathcal O\left(\eps \ln|\ln \de|\, \|\phi\|\right) 
=\begin{cases}
\mathcal O\left(\delta^{2}\,|\ln \de|^2\,(\ln|\ln \de|)^2\right)\ &\hbox{if}\ N=3,\\
o\left(\delta^{N-1}\right)\ &\hbox{if}\ N\ge 4,
\end{cases}\\
&=\begin{cases}
o\left(\delta^{N-2}\right)\ &\hbox{if}\ N=3,\\
o\left(\delta^{N-1}\right)\ &\hbox{if}\ N\ge 4.
\end{cases}  
\end{align*}

\item[($I_7$)] By Hölder's inequality,
\begin{align*} 
&I_7= \io  \left(f'_0(PU_{\delta,\xi} )-f'_0(U_{\delta,\xi} )\right)\phi \,P\psi^j_{\delta,\xi}=\mathcal O\Big(
\big|
\big(f'_0\left(PU_{\delta,\xi}\right)-f'_0\left(U_{\delta,\xi}\right)\big)
P\psi^j_{\delta,\xi}\big |_{{2N\over  N+2}}
|\phi |_{{2N\over N-2}}
\Big).
\end{align*}
By \eqref{y1},
\begin{align*}
|(f'_0\left(PU_{\delta,\xi}\right)-f'_0\left(U_{\delta,\xi}\right))
P\psi^j_{\delta,\xi}|\leq C
\begin{cases}
\left(\delta^{2}+U_{\delta,\xi}^{2^*-3}\delta^{\frac{N-2}{2}}\right)|P\psi^j_{\delta,\xi}| \ &\hbox{if}\ 3\leq N\leq 6,\\
\min\{\delta^{2},U_{\delta,\xi}^{2^*-3}\delta^{\frac{N-2}{2}}\}|P\psi^j_{\delta,\xi}| \ &\hbox{if}\ N\ge 7.\\
\end{cases}
\end{align*}
Note that, by definition $|\psi_{\delta,\xi}^j| \leq (N-2)|U_{\delta,\xi}|$ for $j=0,1,\ldots,N$, and by the maximum principle, there is $C\geq 1$ such that
\begin{align}\label{Ppsi:bd}
|P\psi_{\delta,\xi}^j| \leq C|U_{\delta,\xi}|\quad \text{ in }\Omega,\qquad \text{ for } j=0,1,\ldots,N.
\end{align}
Now we estimate $I_7$ using \eqref{Ppsi:bd} and Lemma \ref{lem:U:q}. If $N\geq 7$ (since $(2^*-2)\frac{2N}{N+2}<\frac{N}{N-2}$),
\begin{align*}
I_7
&=\mathcal O(\|\phi \| \delta^{\frac{N-2}{2}}|U_{\delta,\xi}^{2^*-2}|_{\frac{2N}{N+2}})=\mathcal O(\delta^{\frac{N+2}{2}} \delta^{\frac{N-2}{2}}\delta^{\frac{N-2}{2}(2^*-2)})=\mathcal O(\delta^{N+2})
=o(\delta^{N-1}),
\end{align*}
whereas, if $3\leq N\leq 5 $ (since $(2^*-2)\frac{2N}{N+2}>\frac{N}{N-2}$),
\begin{align*}
I_7&=\mathcal O\Big(\|\phi \| \delta^{2}|U_{\delta,\xi}|_{\frac{2N}{N+2}}
+\|\phi \| \delta^{\frac{N-2}{2}}|U_{\delta,\xi}^{2^*-2}|_{\frac{2N}{N+2}}
\Big)\\
&=\mathcal O\left(\Big(
\delta^{N+\frac{N-2}{2}}
+\delta^{N-2+\frac{N-2}{2}}\delta^{(N-\frac{N-2}{2}(2^*-2)\frac{2N}{N+2})\frac{N+2}{2N}}
\Big)|\ln\delta|(\ln|\ln\delta|)\right)\\
&=\mathcal O\Big((\delta^\frac{3N-2}{2}+\delta^{2(N-2)} 
)|\ln\delta|(\ln|\ln\delta|)\Big)=o(\delta^{N-1}).
\end{align*}
Similarly, if $N=6$ (since $\frac{2N}{N+2}=\frac{N}{N-2}=(2^*-2)\frac{2N}{N+2}=\frac{3}{2}$),
\begin{align*}
I_7&=\mathcal O(
\|\phi \| \delta^{2}|U_{\delta,\xi}|_{\frac{2N}{N+2}}
+\|\phi \| \delta^{\frac{N-2}{2}}|U_{\delta,\xi}^{2^*-2}|_{\frac{2N}{N+2}}
)
\\&=\mathcal O\left(
\Big(
\delta^{N}\delta^{\frac{N+2}{4}}|\ln \delta|^{\frac{N+2}{2N}}
+\delta^{N-2+\frac{N-2}{2}}\delta^{\frac{N}{2}\frac{N+2}{2N}}|\ln\delta|^{\frac{N+2}{2N}}
\Big)|\ln\delta|(\ln|\ln\delta|)
\right)\\
&=\mathcal O\Big((
\delta^{8}+\delta^8)(\ln|\ln\delta|)|\ln \delta|^{\frac{5}{3}}
\Big)=o(\delta^{N-1}).
\end{align*}
In any case, we conclude that $I_7=o(\delta^{N-1})$.

\item[($I_8$)] If $j=0$, by~\eqref{U:q},~\eqref{phibd}, and~\eqref{l2}, 
\begin{align*} 
&\io f'_0(U_{\delta,\xi} )\phi \,(P\psi^0_{\delta,\xi}-\psi^0_{\delta,\xi})\\
&=\mathcal O\left(
\left|f'_0\left(U_{\delta,\xi} \right)\right|_{{N\over  2}}
\left|\phi\right|_{{2N\over N-2}}\left|P\psi^j_{\delta,\xi}-\psi^j_{\delta,\xi}\right|_{{2N\over N-2}}\right)=\mathcal O\left(\delta^{N-2\over2}\|\phi\| \right)\\
& 
=\left.
\begin{cases}
\mathcal O\left( \delta^{\frac32 N-3}|\ln \de|\,\ln|\ln \de|\,\right)\ &\hbox{if}\ 3\le N\le 4,\\
o\left(\delta^{N-1}\right)\ &\hbox{if}\ N\ge 5,\\
\end{cases}
\right\}
=\begin{cases}
o\left( \delta^{N-2}\right)\ &\hbox{if}\ 3\le N\le 4,\\
o\left(\delta^{N-1}\right)\ &\hbox{if}\ N\ge 5.\\
\end{cases}  
\end{align*}

On the other hand, if $j=1,\ldots, N$, by~\eqref{U:q},~\eqref{phibd}, and~\eqref{l2},
\begin{align*} 
&\io f'_0(U_{\delta,\xi} )\phi \,(P\psi^j_{\delta,\xi}-\psi^j_{\delta,\xi})\\
&=\mathcal O\left(
\left|f'_0\left(U_{\delta,\xi} \right)\right|_{{N\over  2}}
\left|\phi\right|_{{2N\over N-2}}\left|P\psi^j_{\delta,\xi}-\psi^j_{\delta,\xi}\right|_{{2N\over N-2}}\right)=\mathcal O\left(\delta^{N\over2}\|\phi\| \right)\\
& 
=\left.
\begin{cases}
\mathcal O\left( \delta^{\frac32 N-2}|\ln \de|\,\ln|\ln \de|\,\right)\ &\hbox{if}\ 3\le N\le 4,\\
o\left(\delta^{N}\right)\ &\hbox{if}\ N\ge 5,\\
\end{cases}
\right\}
=\begin{cases}
o\left( \delta^{N-1}\right)\ &\hbox{if}\ 3\le N\le 4,\\
o\left(\delta^{N}\right)\ &\hbox{if}\ N\ge 5.\\
\end{cases}  
\end{align*}
\end{itemize}

{\em Step 2.}  If $\eps=d \delta^{N-2}|\ln\delta|,$ taking into account all the previous estimates, the left hand side of equation~\eqref{eq:c_i} can be rewritten as
\begin{align}
L.H.S.&= \begin{cases}
\delta^{N-2} F^0_\eps(d,\xi)
&\text{if }\, j=0,\\
\delta^{N-1} F^j_\eps(d,\xi)  &\text{if }\, j\geq 1,\\ 
\end{cases}\label{lhs}
\end{align}
where $F_\eps:[0,+\infty]\times\Omega \to \mathbb R\times\mathbb R^N$ is defined by
$$F^0_\eps(d,\xi):= \mathfrak A_1\;\varrho (\xi)-  \mathfrak A_2 d+o(1)\quad \hbox{and}\quad
F^j_\eps(d,\xi):= \mathfrak A_3 \partial_{\xi_j}\varrho(\xi)+o(1),\; j=1,\dots,N,$$
and the $\mathfrak A_i$'s are positive constants.
\\
We remark that   $F_\eps\to F$ uniformly on compact sets of $[0,+\infty]\times\Omega$ where  the function $F:[0,+\infty]\times\Omega\to \mathbb R\times\mathbb R^N$ is
defined by
$$F(\xi,d)=
\left(\ \mathfrak A_1\;\varrho (\xi)-  \mathfrak A_2 d\ ,\ \mathfrak A_3\nabla \varrho(\xi)\ \right).$$

Now, let
$\xi_0$ be a non-degenerate critical point of the Robin function $\varrho$ and let $d_0:= {\mathfrak A_1\over \mathfrak A_2}\varrho (\xi_0).$
It is easy to check that $(\xi_0,d_0)$ is an isolated zero of $F$
whose Brouwer degree is not zero. Then, if $\eps$ is small enough, there exist $\xi_\eps\to\xi_0$ and $d_\eps\to d_0$ such that $F_\eps(\xi_\eps,d_\eps)=0$. 
Therefore, also the right hand side of  ~\eqref{eq:c_i}
is zero, \textit{i.e.},
$$ \sum_{i=0}^Nc_{\delta,\xi}^i\langle P\psi^{i}_{\delta,\xi},P\psi^j_{\delta,\xi} \rangle=0.$$
Finally, from~\eqref{l3} we immediately deduce that
all the $c_{\delta,\xi}^i$'s are zero. That concludes  the proof.
\end{proof}

\begin{proof}[Proof of Theorem~\ref{main:thm}]
Proposition~\ref{main:prop} implies that $u_\eps=PU_{\delta,\xi}+\phi_{\delta,\xi}$ is a solution of~\eqref{Plog} (see equation \eqref{prob:inverse}).  Moreover, by Lemma~\ref{lem:rey}, Proposition \ref{main:prop}, and Remark~\ref{eps:rate}, statement ~\eqref{uepsbd} holds true, where the function $\Phi_\eps$ in \eqref{uepsbd} is given by
$$
\Phi_\eps=PU_{\delta_\eps,\xi_\eps}-U_{\delta_\eps,\xi_\eps} + \phi_{\delta_\eps,\xi_\eps}.
$$
\end{proof}

\appendix

\section{Appendix: The proof of estimate $I_3$}

This section is devoted to the proof of the following estimate.

\begin{lemma}\label{I3}  As $\delta\to 0$, 
\begin{equation*}
I_3=\io \left(f_0(PU_{\delta,\xi})-f_\eps(PU_{\delta,\xi})\right)\,\psi^j_{\delta,\xi}=
\begin{cases}
-\frac{2d}{N-2} \mathfrak B\;\frac{\eps}{|\ln\delta|}
+ o\big(\frac{\eps}{|\ln\delta|}
\big), &\text{if }\, j=0,\\
\mathcal O\big(\eps\delta^{\frac{N-2}2}\ln|\ln\delta|\big), &\text{if }\, j=1,\ldots,N, 
\end{cases}
\end{equation*}
where
\begin{align}\label{Bval}
\mathfrak B:=-\irn U^p\,[\ln U]\psi^{0} =\frac{\Gamma(\frac{N}{2})\pi^\frac{N}{2}}{4\Gamma(N+1)}N^{\frac{N}{2}}(N-2)^{\frac{N+4}{2}}>0.
\end{align}
\end{lemma}

\begin{remark}\label{rem:1}
\emph{Observe that}        
$$
\eps\delta^{\frac{N-2}{2}}\ln|\ln \de|=d\delta^{\frac32 N-3}\,|\ln \de|\, \ln |\ln \de|=
\begin{cases}
o(\delta^{N-2})\ &\text{if}\ 3\le N\le 4,\\
o(\delta^{N-1})\ &\text{if}\ N\ge 5.
\end{cases}
$$
\end{remark}

\begin{proof}[Proof of Lemma~\ref{I3}]
Taylor's expansion with respect to $\eps$ yields, 
\begin{align*}
&\io \big(f_0(PU_{\delta,\xi}) - f_\eps(PU_{\delta,\xi})\big)\psi_{\delta,\xi}^j\\&=\eps\io (PU_{\delta,\xi})^p\ln\ln(\e+PU_{\delta,\xi})\psi_{\delta,\xi}^j-\eps^2\io \frac{(PU_{\delta,\xi})^p[\ln\ln(\e+PU_{\delta,\xi})]^2\psi_{\delta,\xi}^j}{       1+\eps\ln(\e+PU_{\delta,\xi})     },
\end{align*}
because, by~\eqref{bub}, 
and Lemma~\ref{lem:A1}, it holds that
\begin{align*}
&\left|\io \frac{(PU_{\delta,\xi})^p(x) [\ln\ln(\e+PU_{\delta,\xi}(x))]^2\psi_{\delta,\xi}^j(x)}{       1+\eps\ln(\e+PU_{\delta,\xi})     }\d x\right|\\
&\leq\io U_{\delta,\xi}^p(x)[\ln\ln(\e+U_{\delta,\xi}(x))]^2|\psi_{\delta,\xi}^j(x)|\d x=\mathcal O\big((\ln|\ln \de|)^2\big). 
\end{align*}
Next, we set $g(u):=u^p\ln\ln(\e+u)$. Then, the mean value theorem yields
$$       0\le     g(u)-g(v)\leq C\,        u^{p-1}\,\left(\ln\ln(\e+u) + 1\right)     [u-v],\qquad\text{if }0\leq v\leq u.$$
It follows that
$$\io PU^p_{\delta,\xi}[\ln\ln(\e + PU_{\delta,\xi})]\psi_{\delta,\xi}^j=\io U^p_{\delta,\xi}[\ln\ln(\e + U_{\delta,\xi})]\psi_{\delta,\xi}^j +        \mathcal O\big(\delta^{\frac{N-2}{2}}\ln|\ln \de|\big)    ,$$
because, by~\eqref{bub},~\eqref{ppsi}, and Lemmas~\ref{lem:A1} and ~\ref{lem:rey}, it holds, for $\de\in(0,\de_0)$ small, that
\begin{align*}
&\io  U_{\delta,\xi}^{p-1}(x) \big(\ln\ln(\e+U_{\delta,\xi}(x))+1\big)     [U_{\delta,\xi}(x)-PU_{\delta,\xi}(x)]|\psi_{\delta,\xi}^j(x)|\\
&\leq C       \delta^{N+\frac{N-2}{2}-\frac{N-2}{2}-\frac{N-2}{2}(p-1)}(\ln|\ln \de|+1)\int_{\R^N} U^{p-1}(y)\,\big(H(\xi+\delta y,\xi)+\mathcal O(\de)\big) |\psi^j(y)|\\
&=C \delta^{\frac{N-2}{2}}(\ln|\ln \de|+1) \int_{\R^N} U^{p-1}(y) \,\big(H(\xi+\delta y,\xi)+\mathcal O(\de)\big) |\psi^j(y)|\\
&\leq C  \delta^{\frac{N-2}{2}}\ln|\ln \de| \left(\irn U^{p-1}(y) \,|\psi^j(y)|\,\d y\right)\, \big(H(\xi,\xi)+\mathcal O(\de)\big).
\end{align*}

Now, using~\eqref{bub},~\eqref{ppsi}, Lemma~\ref{lem:A1}, and statement~\eqref{eq:Upsi}, we get
\begin{align}\label{Up:ln:psi:j}
&\io U^p_{\delta,\xi}(x)[\ln\ln(\e + U_{\delta,\xi}(x))]\psi_{\delta,\xi}^j(x)\d x  \nonumber\\
&=\delta^{N-\frac{N+2}{2}-\frac{N-2}{2}} \int_{\Omega_\delta}U^p(y)[\ln\ln(\e + \delta^{-\frac {N-2}2}U(y))] \psi^j(y)\d y \nonumber\\
&=\ln\ln(\delta^{-\frac{N-2}{2}}) \int_{\Omega_\delta}U^p(y)\psi^j(y)dy \nonumber\\
&\quad + \frac{1}{|\ln\delta|}\int_{\Omega_\delta}U^p(y)\left(|\ln\delta|\,\ln\left[1+\frac{\ln(\e^{1-\frac{N-2}{2}|\ln\delta|}+U(y))}{\frac{N-2}{2}|\ln\delta|}\right]\right) \psi^j(y)dy.
\end{align}
Note that, for $j=1,\ldots,N$, the function 
\begin{align*}
y\mapsto \varphi(y):=U^p(y)\left(|\ln\delta|\,\ln\left[1+\frac{\ln(\e^{1-\frac{N-2}{2}|\ln\delta|}+U(y))}{\frac{N-2}{2}|\ln\delta|}\right]\right) \psi^j(y)
\end{align*}
is odd. Hence, its integral over $\rn$ is equal to zero and, by~\eqref{shlim},
\begin{align*}
&\int_{\R^N}\varphi(y)-\int_{\Omega_\delta}\varphi(y)
=\int_{\R^N\backslash \Omega_\delta}\varphi(y)\\
&\leq C\left(|\ln\delta|\,\ln\left[1+\frac{\ln(\e^{1-\frac{N-2}{2}|\ln\delta|}+\alpha_N)}{\frac{N-2}{2}|\ln\delta|}\right]\right)\int_{\R^N\backslash \Omega_\delta} U^p(y) |\psi^j(y)| \\
&\le C\left(\frac{2}{N-2}\ln \alpha_N+o(1)\right)\delta^{N+1}
=\mathcal O(\delta^{N+1}).
\end{align*}       
Consequently,
\begin{equation*}
\io U^p_{\delta,\xi}(x)[\ln\ln(\e + U_{\delta,\xi}(x))]\psi_{\delta,\xi}^j(x)\d x = \mathcal  O\big(\delta^{N+1}\ln|\ln(\delta)|\big)  =o(\delta^{N}),
\end{equation*}
and the claim concerning $j=1,\ldots,N$ follows.

On the other hand, for $j=0$ we also have $\irn U^p\psi^0=0$ (see~\eqref{eq:Upsi}), and  it can also be proved that   $\int_{\R^N\backslash \Omega_\delta}U^p\psi^0 = \mathcal O(\delta^{N})$. Then, by \eqref{Up:ln:psi:j} and~\eqref{shlim},      
\begin{align*}
&\io U^p_{\delta,\xi}(x)[\ln\ln(\e + U_{\delta,\xi}(x))]\psi_{\delta,\xi}^0(x)\d x  \\
&= \frac{1}{|\ln\delta|}\,\frac{2}{N-2}\irn U^p(y)\,[\ln U(y)]\,\psi^0(y)dy  + 
o\left(\frac{1}{|\ln\delta|}\right),
\end{align*}
as claimed. 
To finish the proof, we show~\eqref{Bval}. Indeed, passing to polar coordinates, integrating by parts, changing variables ($s=r^2$, $dr = \frac{1}{2}s^{-\frac{1}{2}}\, ds$), and using that $|\partial B_1(0)|=\frac{2\pi^\frac{N}{2}}{\Gamma(\frac{N}{2})}$,
\begin{align*}
\mathfrak B&=-|\partial B_1(0)|\frac{\alpha_N^{2^*}(N-2)}{2}\int_0^\infty \frac{r^{N-1}}{(1+r^2)^{\frac{N+2}{2}}}\ln\left(\frac{1}{(1+r^2)^\frac{N-2}{2}}\right)\frac{r^2-1}{(1+r^2)^\frac{N}{2}}\ dr\\
&=\frac{\alpha_N^{2^*}(N-2)^2}{4}\frac{2\pi^\frac{N}{2}}{\Gamma(\frac{N}{2})}\int_0^\infty \frac{r^{N-1}(r^2-1)}{(1+r^2)^{N+1}}\ln(1+r^2)\ dr\\
&=\frac{\alpha_N^{2^*}(N-2)^2\pi^\frac{N}{2}}{2\Gamma(\frac{N}{2})}\int_0^\infty\frac{1}{N}\left(\frac{r}{1+r^2}\right)^N\frac{2r}{1+r^2}\ dr\\
&=\frac{\alpha_N^{2^*}(N-2)^2\pi^\frac{N}{2}}{2N\Gamma(\frac{N}{2})}\int_0^\infty\frac{s^{\frac{N}{2}}}{(1+s)^{N+1}}\ ds
=\frac{\alpha_N^{2^*}(N-2)^2\pi^\frac{N}{2}}{2N\Gamma(\frac{N}{2})}B\left(\frac{N}{2}+1,\frac{N}{2}\right),\\
&=\frac{\alpha_N^{2^*}(N-2)^2\pi^\frac{N}{2}}{2N\Gamma(\frac{N}{2})}\frac{\Gamma(\frac{N}{2}+1)\Gamma(\frac{N}{2})}{\Gamma(N+1)}
=\frac{\Gamma(\frac{N}{2})\pi^\frac{N}{2}}{4\Gamma(N+1)}N^{\frac{N}{2}}(N-2)^{\frac{N+4}{2}},
\end{align*}
where $B(\cdot,\cdot)$ denotes the usual Beta function.
\end{proof}

\section{Appendix: Further estimates}

\begin{lemma}${}$ \label{lem:meanvalue}
\begin{itemize}
\item[$(i)$] $|f_\eps(u)-f_0(u)|\leq \eps\,|u|^{2^*-1}\, \ln\ln(\e + |u|)$.
\item[$(ii)$] For $\eps$ small enough,         and any $u\in\r$,
\begin{align}\label{f':eps}
|f'_\eps(u)|&\le C|u|^{2^*-2},
\end{align}
and      
$$|f'_{\eps}(u)-f'_0(u)| \leq \eps |u|^{2^*-2}\left(       (2^*-1)     \ln \ln(\e+|u|) + \frac{1}{\ln(\e+|u|)}\right).$$
\item[$(iii)$] There exists $C>0$ such that, for $\eps$ small enough and any $u,v\in\r$,
\begin{equation}\label{f'eps:N}
|f'_\eps(u+v)-f'_\eps(u)|\leq\begin{cases}
C(|u|^{2^*-3}+|v|^{2^*-3})|v| &\text{if }N\leq 6,\smallskip \\
C(|v|^{2^*-2}+\eps|u|^{2^*-2})     &\text{if }N>6.
\end{cases}
\end{equation}
\end{itemize}
\end{lemma}

\begin{proof}
$(i):$ Since
\begin{align*}
\frac{\partial f_\eps(u)}{\partial\eps}&=-\frac{|u|^{2^*-2}u}{[\ln(\e+|u|)]^\eps}\,\ln \ln(\e+|u|),
\end{align*}
we have that
\begin{align*}
\left|\frac{\partial f_\eps(u)}{\partial\eps}\right|&\leq |u|^{2^*-1}\ln \ln(\e+|u|),
\end{align*}
and the statement follows easily from the mean value theorem.

$(ii):$ As
\begin{align}\label{def:f':eps}
f'_\eps(u)&=\frac{|u|^{2^*-2}}{[\ln(\e+|u|)]^\eps}\left(2^*-1-\frac{\eps|u|}{(\e+|u|)\ln(\e+|u|)}\right),
\end{align}
then~\eqref{f':eps} holds
due to $\  0\le\frac{|u|}{\e+|u|}\le 1$, and  $\ 0\le\frac{1}{\ln(\e+|u|)}\le 1.$

Since~\eqref{def:f':eps}, we have that
\begin{align*}
\frac{\partial f'_\eps(u)}{\partial\eps}&=-\frac{|u|^{2^*-2}\ln \ln(\e+|u|)}{[\ln(\e+|u|)]^\eps}\left(2^*-1-\frac{\eps|u|}{(\e+|u|)\ln(\e+|u|)}\right)\\
&\quad - \frac{|u|^{2^*-1}}{(\e+|u|)[\ln(\e+|u|)]^{\eps +1}}.
\end{align*}
Hence, for $\eps$ small enough,
\begin{align*}
\left|\frac{\partial f'_\eps(u)}{\partial\eps}\right|&\leq        (2^*-1)     \frac{|u|^{2^*-2}\ln \ln(\e+|u|)}{[\ln(\e+|u|)]^\eps} + \frac{|u|^{2^*-2}}{[\ln(\e+|u|)]^{\eps +1}} \\
&\leq |u|^{2^*-2}\left(       (2^*-1)     \ln \ln(\e+|u|) + \frac{1}{\ln(\e+|u|)}\right).
\end{align*}
Now the claim follows from the mean value theorem.

$(iii):$ Setting $p:=2^*-1$, we see that
\begin{align*}
&f''_\eps(u)=\frac{\eps\,|u|^{2^*-3}u}{[\ln(\e+|u|)]^\eps}
\left(\frac{|u|-\e\ln(\e+|u|)}{(\e+|u|)^2(\ln(\e+|u|))^2}\right)\\
&+\frac{|u|^{2^*-4}u}{[\ln(\e+|u|)]^\eps}
\left(p-1-\frac{\eps|u|}{(\e+|u|)\ln(\e+|u|)}\right)
\left(p-\frac{\eps|u|}{(\e+|u|)\ln(\e+|u|)}\right).
\end{align*}
So, for $\eps$ small enough,
\begin{align*}
&|f''_\eps(u)|\le C|u|^{2^*-3}.
\end{align*}
Since $2^*-3\geq 0$ for $N\leq 6$, the mean value theorem yields
$$|f'_\eps(u+v)-f'_\eps(u)|=|f''_\eps(u+tv)||v|\leq C|u+tv|^{2^*-3}|v|\leq C(|u|^{2^*-3}+|v|^{2^*-3})|v|$$
for some $t\in(0,1)$, as stated in $(iii)$ for $N\leq 6$.

Next, assume $N>6$. Then, $q:=2^*-2\in(0,1)$. We write
\begin{align*}
&|f'_\eps(u+v)-f'_\eps(u)|       \le     p\left|\frac{|u+v|^q}{[\ln(\e+|u+v|)]^\eps}-\frac{|u|^q}{[\ln(\e+|u|)]^\eps}\right|\\
&\qquad + \eps\left|\frac{|u+v|^p}{(\e+|u+v|)[\ln(\e+|u+v|)]^{\eps+1}}-\frac{|u|^p}{(\e+|u|)[\ln(\e+|u|)]^{\eps+1}}\right|\\
&\quad=:F_1+F_2.
\end{align*}
Clearly, 
$$F_2\leq        C\eps(|u|^q+|v|^q).     
$$
Now, for any fixed $v\neq 0$, setting $x:=\frac{u}{v}$ we have that
\begin{align*}
&\frac{1}{|v|^q}\left|\frac{|u+v|^q}{[\ln(\e+|u+v|)]^\eps}-\frac{|u|^q}{[\ln(\e+|u|)]^\eps}\right|\\
&\qquad=\left|\frac{|x+1|^q}{[\ln(\e+|v||x+1|)]^\eps}-\frac{|x|^q}{[\ln(\e+|v||x|)]^\eps}\right|=:g(x).
\end{align*}
The function $g$ is symmetric with respect to $-\frac{1}{2}$, \textit{i.e.}, $g(x)=g(-x-1)$, it is increasing in $[-\frac{1}{2},0]$ and decreasing in $[0,\infty)$, and $g(0)\leq 1$. Hence,
$$F_1\leq p|v|^q.$$
This proves~\eqref{f'eps:N}, concluding the proof of statement $(iii)$. 
\end{proof}

\begin{lemma} \label{lem:A1}
Let $r>0$. Then, for any $u>0$ and $\delta\in(0,1)$,
$$\ln \ln \left(\e+\delta^{-r}u\right) = \ln\ln \left(\delta^{-r}\right) + \ln\left(1+\frac{\ln(\e^{1-r|\ln\delta|} +u)}{r|\ln\delta|}\right),$$
and
\begin{align}\label{shlim}
\lim_{\delta\to 0}\left(|\ln\delta|\,\ln\left[1+\frac{\ln(\e^{1-r|\ln\delta|} +u)}{r|\ln\delta|}\right]\right)=\frac{1}{r}\ln u.
\end{align}
\end{lemma}

\begin{proof}
We have that
\begin{align*}
&\ln \ln \left(\e+\delta^{-r}u\right) = \ln\ln(\delta^{-r}(\delta^r\e + u))=\ln\left(\ln \delta^{-r}+\ln(\e^{1-r|\ln\delta|}+u)\right)\\
&=\ln\left[\ln \delta^{-r}\left(1+\frac{\ln(\e^{1-r|\ln\delta|}+u)}{\ln \delta^{-r}}\right)\right]\\
&= \ln\ln \left(\delta^{-r}\right) + \ln\left(1+\frac{\ln(\e^{1-r|\ln\delta|} +u)}{r|\ln\delta|}\right).
\end{align*}
Set
\begin{equation*}
g(t):=\ln\left[1+\frac{t}{r}\ln(\e^{1-\frac{r}{t}}+u)\right],\qquad t>0.
\end{equation*}
Applying L'Hôpital's rule we obtain
\begin{align*}
\lim_{t\to 0}\frac{g(t)}{t}&=\lim_{t\to 0}g'(t)=\lim_{t\to 0}\frac{\frac{1}{r}\ln(\e^{1-\frac{r}{t}}+u)+\frac1{t} \e^{1-\frac{r}{t}}(\e^{1-\frac{r}{t}}+u)^{-1}}{1+\frac{t}{r}\ln(\e^{1-\frac{r}{t}}+u)}\\
&=\frac{1}{r}\ln u.
\end{align*}
Taking $t:=|\ln\delta|^{-1}$, we obtain the claim.
\end{proof}

\bigskip

\begin{flushleft}
\textbf{Mónica Clapp and Alberto Saldaña}\\
Instituto de Matemáticas,\\
Universidad Nacional Autónoma de México,\\
Circuito Exterior, Ciudad Universitaria,\\
04510 Coyoacán, Ciudad de México, Mexico.\\
\texttt{monica.clapp@im.unam.mx,\; alberto.saldana@im.unam.mx} 
\medskip

\textbf{Rosa Pardo}\\
Departamento de Análisis Matemático y Matemática Aplicada,\\
Facultad de Ciencias Químicas,\\
Universidad Complutense de Madrid,\\
28040 Madrid, Spain.\\
\texttt{rpardo@ucm.es}
\medskip

\textbf{Angela Pistoia}\\
Dipartimento di Metodi e Modelli Matematici,\\
La Sapienza Università di Roma,\\
Via Antonio Scarpa 16,\\
00161 Roma, Italy.\\
\texttt{angela.pistoia@uniroma1.it} 
\end{flushleft}

\end{document}